\documentclass[sts]{imsart}

\usepackage{geometry}
\usepackage{amssymb,amsbsy,amsmath,amscd,amsthm,amsfonts,MnSymbol}
\usepackage{cite}
\usepackage[utf8]{inputenc}
\usepackage{enumerate,hyperref}
\usepackage{dsfont}
\usepackage{graphicx}
\usepackage{tikz}

\newtheorem{definition}{Definition}
\newtheorem{theorem}{Theorem}
\newtheorem{proposition}{Proposition}

\newtheorem{example}{Example}
\newtheorem{lemma}{Lemma}
\newtheorem{remark}{Remark}

\newtheorem{corollary}{Corollary}

\newtheorem*{openquestion}{Open question}

\newcommand{\R}{\mathbb R}
\newcommand{\PP}{\mathbb P}
\newcommand{\E}{\mathbb E}

\newcommand{\N}{\mathbb N}

\newcommand{\tr}{\mathrm{tr}}

\newcommand{\var}{\mathrm{var}}

\newcommand{\cov}{\mathrm{cov}}

\newcommand*\diff{\mathop{}\!\mathrm{d}}

\newcommand*\dist{\mathop{}\!\mathrm{d}}

\newcommand{\inter}{\mathrm{int}}

\newcommand{\cl}{\mathrm{cl}}

\newcommand{\DS}{\displaystyle}

\DeclareMathOperator*{\Argmin}{Argmin}

\DeclareMathOperator*{\Argmax}{Argmax}

\DeclareMathOperator*{\argmin}{argmin}

\DeclareMathOperator*{\argmax}{argmax}

\begin{document}

\begin{frontmatter}

\title{Asymptotics of constrained $M$-estimation under convexity}
\runtitle{Asymptotics of convex $M$-estimation}

\author{Victor-Emmanuel Brunel \thanks{CREST-ENSAE, victor.emmanuel.brunel@ensae.fr}}

\runauthor{V.-E. Brunel}

\setattribute{abstractname}{skip} {{\bf Abstract:} } 
\begin{abstract}
$M$-estimation, \textit{aka} empirical risk minimization, is at the heart of statistics and machine learning: Classification, regression, location estimation, etc. Asymptotic theory is well understood when the loss satisfies some smoothness assumptions and its derivatives are dominated locally. However, these conditions are typically technical and can be too restrictive or heavy to check. Here, we consider the case of a convex loss function, which may not even be differentiable: We establish an asymptotic theory for $M$-estimation with convex loss (which needs not be differentiable) under convex constraints. We show that the asymptotic distributions of the corresponding $M$-estimators depend on an interplay between the loss function and the boundary structure of the set of constraints. 
We extend our results to $U$-estimators, building on the asymptotic theory of $U$-statistics. Applications of our work include, among other, robust location/scatter estimation, estimation of deepest points relative to depth functions such as Oja's depth, etc. 

\end{abstract}


\end{frontmatter}

\section{Introduction}

\subsection{Preliminaries} \label{sec:Preliminaries}
We consider a sequence $X_1,X_2,\ldots$ of independent, identically distributed (iid) random variables taking values in some measurable space $(E,\mathcal E)$ and we denote by $P$ their distribution. Let $\Theta_0\subseteq \R^d$ be a non-empty set, which can be interpreted as a parameter space. Here, $d\geq 1$ is a fixed integer representing the parameter dimension. 

Let $\phi:E\times\Theta_0\to\R$ be a function such that $\phi(\cdot,\theta)$ is measurable and in $L^1(P)$, for all $\theta\in\Theta_0$. Set $\Phi(\theta)=\E[\phi(X_1,\theta)]$, for all $\theta\in\Theta_0$. The goal of $M$-estimation (or empirical risk minimization) is to estimate a minimizer of $\Phi$ when only finitely many samples from $P$ are available. 
For $n\geq 1$ and $\theta\in\Theta_0$, let $\DS \Phi_n(\theta)=\frac{1}{n}\sum_{i=1}^n \phi(X_i,\theta)$. For $\theta\in\Theta$, $\Phi(\theta)$ is called the population risk evaluated at $\theta$, while $\Phi_n(\theta)$ is the empirical risk based on $X_1,\ldots,X_n$. The idea of $M$-estimation is to use the random function $\Phi_n$ as a surrogate for $\Phi$ and estimate a minimizer of $\Phi$ by selecting a minimizer of $\Phi_n$. When minimization is performed over the whole parameter space $\Theta_0$, we talk about unconstrained $M$-estimation, or simply $M$-estimation. If we minimize $\Phi_n$ on a closed subset $\Theta$ of $\Theta_0$, we talk about constrained $M$-estimation with $\Theta$ as the set of constraints. In this work, we are concerned with the latter.

Let $\Theta^*\subseteq \Theta$ be the set of minimizers of $\Phi$ on $\Theta$ and assume it is not empty. For all $n\geq 1$, let $\hat \theta_n$ be a minimizer of $\Phi_n$ (provided it exists and can be chosen in a measurable way - see Section~\ref{Sec:existence} below). Standard asymptotic theory questions (weak or strong) consistency and aims at determining the asymptotic distribution of a rescaled version of the $M$-estimator. That is, does $\dist(\hat\theta_n,\Theta^*)$ converge (in probability or almost surely) to zero as $n\to\infty$? Here, $\dist(\hat\theta_n,\Theta^*)$ is simply the distance of $\hat\theta_n$ to the non-empty set $\Theta^*$. If $\Theta^*$ reduces to a singleton $\Theta^*=\{\theta^*\}$, does $\sqrt \rho_n(\hat\theta_n-\theta^*)$ converge in distribution for some rescaling factor $\rho_n\xrightarrow[n\to\infty]{}\infty$ and if so, what is the asymptotic distribution?

It may be convenient to consider, instead of $\hat\theta_n$, a near minimizer of $\Phi_n$, that is, a random variable $\tilde\theta_n$ satisfying $\Phi_n(\tilde\theta_n)\leq \inf_{\theta\in\Theta}\Phi_n(\theta)+\varepsilon_n$ where $\varepsilon_n$ is a (possibly random) small enough error term. For simplicity, here, we only study the properties of exact empirical risk minimizers.  

Our main working assumption is that the loss function is convex in its second argument. That is, $\Theta_0$ and $\Theta$ are convex sets and $\phi(x,\cdot)$ is convex on $\Theta_0$ for $P$-almost all $x\in E$. Relevant examples include:
\begin{enumerate}
	\item Location estimation: $E=\Theta_0=\R^d$, $\phi(x,\theta)=\ell(x-\theta)$ for some convex function $\ell:\R^d\to\R$. For instance, if $\ell$ is the squared Euclidean norm, we recover mean estimation. If $\ell$ is the Euclidean norm, we recover geometric median estimation. If $\ell(x)=\|x\|-(1-2\alpha)u^\top x$, where $\alpha\in (0,1)$ and $u\in\R^d$ with $\|u\|=1$ are fixed ($\|\cdot\|$ being the Euclidean norm), we recover geometric quantile estimation (e.g., if $d=1$ and $u=1$, $\Theta^*$ is simply the set of $\alpha$-quantiles of $P$). Huber's $M$-estimators, adding robustness to mean estimators, correspond to the loss $\ell(x)=h_c(\|x\|), x\in\R^d$, where for all $t\geq 0$, $h_c(t)=t^2$ if $t\leq c$, $h_c(t)=2ct-c^2$ if $t>c$ and $c>0$ is a given, tuning parameter.

	\item Location estimation on matrix spaces: Let $E=\Theta_0=:\mathcal S_d^+$ be the space of $d\times d$ symmetric, positive semi-definite matrices. There are several ways of averaging positive definite matrices, beyond simply taking their arithmetic mean (i.e., their standard linear average). A simple example is that of the harmonic mean, which is simply the inverse of the linear average of the inverses (if the matrices are positive definite). More involved ways include (again for positive definite matrices) the Karsher mean, which, in the case of $2$ such matrices, reduces to their geometric mean \cite{bhatia2006riemannian}. In the context of optimal transport, a large body of literature has been interested in the Bures-Wasserstein mean of positive definite matrices, which is related to Wasserstein barycenters on the set of Gaussian distributions \cite{takatsu2011wasserstein,altschuler2021averaging}. In fact, it is shown in \cite[Lemma A.5]{kroshnin2021statistical} that the Bures-Wasserstein mean is the solution to a convex optimization problem. Hence, as it is done in \cite{kroshnin2021statistical}, the Bures-Wasserstein barycenter of iid, random, positive (semi-)definite matrices can be analyzed under the prism of $M$-estimation with convex loss, and our results also allows to consider the constrained case, as well as robust alternatives to Bures-Wasserstein barycenters (such as the Bures-Wasserstein median, see \cite{altschuler2021averaging}). 
		
	\item Linear regression (here, data are rather denoted as pairs $(X_n,Y_n)\in\R^d\times\R, n\geq 1$): $E=\R^d\times\R$, $\Theta=\R^d$, $\phi((x,y),\theta)=\ell(y-\theta^\top x)$ for some $\ell:\R\to\R$ (which, again in our context, we assume to be convex). If $\ell(t)=t^2$, we recover least squares estimation. If $\ell(t)=|t|$, this is median regression, etc.		
\end{enumerate}

In all these examples, we can take $\Theta_0=\Theta=\R^d$ (or $\mathcal S_d^+$), corresponding to unconstrained estimation, but we could also assume that $\Theta$ is a closed, strict subset of $\Theta_0$. Perhaps the simplest example is the case when $E=\Theta_0=\R^d$, $\Theta\subseteq\R^d$ is a compact convex subset and $\phi(x,\theta)=\|x-\theta\|^2$. In that case, it is easy to check that $\theta^*=\pi_{\Theta}(\E[X])$ and $\hat\theta_n=\pi_\Theta(\bar X_n)$  are the unique minimizers of $\Phi$ and $\Phi_n$ respectively, where $\bar X_n=n^{-1}\sum_{i=1}^n X_i$ and $\pi_\Theta$ is the metric projection on $\Theta$. Of course, this example can be studied with elementary tools, but it is worth keeping it in mind as an illustration of our results, in order to fix ideas.

Typically, proving consistency and finding the asymptotic distribution of $M$-estimators require some tools from the theory of empirical processes and imposes some smoothness of the loss function $\phi$ in its second argument. Moreover, it is often assumed that the partial derivatives of $\phi$, with respect to its second argument, are locally dominated, allowing the use of dominated convergence to swap derivatives and expectations in the analysis. In our context, the full power of convexity comes in through fairly elementary convex analysis and allows to completely avoid such common technical assumptions. 


\subsection{Related works}

$M$-estimation is a quintessential problem in statistical inference (maximum likelihood estimation being a particular instance in general) and, as a particular case, constrained $M$-estimation. 

Asymptotic theory of statistical estimation has been overlooked in the era of high-dimensional data and models. Yet, it provides benchmarks for non-asymptotic theory and asymptotic approximations produce less conservative inference than non-asymptotic approaches, and they are relevant when the data set contains a lot of samples and their dimension is not too large. 

Asymptotic theory of $M$-estimators is well understood when the loss function is smooth and satisfies local domination properties \cite{lecam1970assumptions,van1996weak,van2000asymptotic}.
Under similar smoothness and domination assumptions, \cite{geyer1994asymptotics} also derived asymptotic properties in the constrained case, when the set of constraints is a regular closed set and the population minimizer is a local minimum of the population risk in the ambient space. See also \cite{li2024inference} for inference on constrained statistical problems and \cite{pollard1991asymptotics,knight2006asymptotic} for special cases. Recently, \cite{li2015geometric} drew connections between the statistical error of constrained $M$-estimation and the statistical dimension of the constrained set, building on \cite{chatterjee2014new,plan2017high} in linear regression and Gaussian sequence models. Even though these connections belong to the non-asymptotic world, we also discuss such connections at infinitesimal scales in the remarks following Theorem~\ref{thm:asympt_norm} below.

When the loss function is convex, \cite{haberman1989concavity} proved asymptotic normality, only requiring the population risk (that is, $\Phi$) being twice differentiable at the (unique) population minimizer, with positive definite Hessian at that point - convexity allowing to avoid any local domination assumption. \cite{niemiro1992asymptotics} proved further asymptotic expansions of the statistical error under stronger smoothness assumptions of convex the loss. 

Asymptotics of penalized $M$-estimators have also been established \cite{honorio2014unified}, in particular for penalized regression (such as Lasso) \cite{knight2000asymptotics}.

In the context of high dimensional linear regression and classification, some recent work has also tackled the asymptotics of penalized $M$-estimators and bagged penalized $M$ estimators in growing dimension (that is, when the dimension $d$ also diverges with the sample size) \cite{bellec2022asymptotic,koriyama2024precise,bellec2024asymptotics}. Related to this line of work are the high-dimensional central limit theorems of \cite{chernozhukov2017central,fang2021high} which correspond to the squared Euclidean loss in the context of $M$-estimation. To the best of our knowledge, similar high-dimensional central limit theorems have not been tackled for general $M$-estimators. 

This work is not concerned with penalized $M$-estimation. Indeed, even though penalized and constrained optimization problems are related through Lagrangian functions, in penalized statistical problems, it is standard to let the penalty depend on the sample size in order to enforce some regularization and achieve optimal performance, although here, we only consider fixed constraint sets, independently of the sample size. 

\subsection{Outline}

In Section~\ref{Sec:lemmas}, we give some key lemmas that we use in our main results. Section~\ref{sec:lemmas1} gathers some results about convex functions and sequences of convex functions, which we chose to highlight in the first part of this work because they are essential to build the intuition behind the theory. In Section~\ref{Sec:existence}, which is much more theoretical and could be skipped at first, we deal with the existence of a measurable empirical minimizer, based on results that guarantee the existence of measurable selections. Section~\ref{Sec:consistency} focuses on consistency of convex $M$-estimators and Section~\ref{Sec:asymptnorm} deals with asymptotic distributions of $M$-estimators. We propose an extension to $U$-estimators with convex loss in Section~\ref{sec:U-estimators}. More lemmas about convex functions, convex sets and cones, and metric projections, which are only used for some technical parts of the main proofs, but not essential to build the intuition, are deferred to the appendix. However, Section~\ref{sec:appendix_diff_proj}, in the appendix, on directional differentiability of metric projections onto convex sets, may be of independent interest to the reader.

\subsection{Notation and standard definitions/assumptions}

Here, we gather all the notation that we use in this work, as well as several simple definitions.

\begin{enumerate}

	\item In this work, $(\Omega,\mathcal F,\PP)$ is a fixed probability space and we assume that all the random variables that we consider are defined on that space. We let $X_1,X_2,\ldots$ be iid random variables with values in a measurable space $E$ and we let $P=X_1\#\PP$ be their distribution. The set $\Theta_0$ is a fixed, open, convex subset of $\R^d$ and $\Theta$ is a closed, convex subset of $\Theta_0$. The loss function $\phi:E\times\Theta_0\to\R$ is assumed to be measurable in its first argument and convex in its second, and to satisfy $\phi(\cdot,\theta)\in L^1(P)$ for all $\theta\in\Theta_0$. We let $\Phi(\theta)=\E[\phi(X_1,\theta)]$ for all $\theta\in\Theta_0$ (referred to as \textit{population risk}) and for all $n\geq 1$, $\omega\in\Omega$ and $\theta\in\Theta_0$, $\Phi_n(\omega,\theta)=n^{-1}\sum_{i=1}^n\phi(X_i(\omega),\theta)$ (referred to as \textit{empirical risk}). For simplicity, unless this amount of precision is needed, we simply write $\Phi_n(\theta)$ and skip the dependence on $\omega\in\Omega$. 
	
	\item The power set of a non-empty set $A$ is denoted by $\mathcal P(A)$. 

	\item Given a subset $G\subseteq\R^d$, we denote by $\mathrm{int}(G)$ its interior, $\cl(G)$ its closure and $\partial G=\cl(G)\setminus\inter(G)$ its boundary.

	\item Any symmetric, positive definite matrix $S\in\R^{d\times d}$ yields a scalar product by setting, for $x,y\in\R^d$, $\langle x,y\rangle_S:=x^\top S y$. The associated Euclidean norm is given by $\|x\|_S=\langle x,x\rangle_S^{1/2}$ for all $x\in\R^d$. The corresponding Euclidean ball with center $x\in \R^d$ and radius $r\geq 0$ is denoted by $B_S(x,r)$. 

	\item Given a vector $u\in\R^d$, the linear subspace of $\R^d$ that is orthogonal to $u$ with respect to $\langle\cdot,\cdot\rangle_S$ is denoted by $u^{\perp_S}$: If $u=0$, $u^{\perp_S}=\R^d$ and if $u\neq 0$, $u^{\perp_S}$ is some linear hyperplane. When $L\subseteq \R^d$, we denote by $L^{\perp_S}$ the linear subspace of $\R^d$ that is orthogonal to $L$ with respect to $\langle\cdot,\cdot\rangle_S$. 

	\item For a set $C\subseteq \R^d$, a vector $u\in\R^d$ and a real number $t\in\R$, we denote by $C_{u,t}^S=\{x\in C:\langle u,x\rangle_S=t\}$, which may be empty. When $t=0$, we simply write $C_u^S=C_{u,t}^S$. 

	\item The distance of a point $x\in\R^d$ to a closed set $C\subseteq \R^d$ with respect to the Euclidean norm associated with $S$ is denoted by $\dist_S(x,C)=\min_{y\in C}\|x-y\|_S$.  

	\item The metric projection onto a non-empty, closed convex set $C\subseteq \R^d$ with respect to $\langle\cdot,\cdot\rangle_S$ is denoted by $\pi_C^S$: For all $u\in\R^d$, $\pi_C^S(u)$ is the unique minimizer of the map $t\in C\mapsto \|t-u\|_S^2$. In particular, $\dist_S(u,C)=\|u-\pi_C^S(u)\|_S$.

	\item Let $G\subseteq\R^d$ be a non-empty, closed, convex set and $x_0\in G$. The tangent cone to $G$ at $x_0$ is the set of all $t\in\R^d$ such that $x_0+\varepsilon t\in G$ for all small enough $\varepsilon>0$. It is a convex cone, not necessarily closed. Its closure is called the support cone to $G$ at $x_0$. Let $S\in\R^{d\times d}$ be symmetric, positive definite. The normal cone to $G$ at $x_0$ with respect to $S$ is the set of all $t\in\R^d$ satisfying $\langle t,x-x_0\rangle_S\leq 0$ for all $x\in G$. It is a closed, convex cone. When there is no mention of a matrix $S$, it is implicitly assumed to be the identity matrix. 

	\item The support function of a non-empty convex set $C\subseteq\R^d$ is the map $h_C:\R^d\to\R\cup\{\infty\}$ defined by $h_C(t)=\sup_{u\in C} u^\top t$. If $t\neq 0$, it is the largest (signed) distance from the origin to a hyperplane orthogonal to $t$ and that is tangent to $C$. It is easy to check that $h_C$ is a sublinear function (that is, positively homogeneous and convex). If $C$ is bounded, then $h_C$ only takes finite values. See, e.g., \cite[Section 1.7.1]{schneider2014convex}.

	\item In all notation above, when $S$ is the identity matrix, we drop the subscript or superscipt $S$ and simply write, for instance, $\|x\|$, $B(x,r)$, $u^\perp$, $C_u$, $\pi_C$, etc.

	\item Given a set $C\subseteq \R^d$ and a function $f:C\to\R$, the set of minimizers (resp. maximizers) of $f$ on $C$ is denoted by $\Argmin_{y\in C} f(y)$ (resp. $\Argmax_{y\in C} f(y)$). This set may be empty. When this set is a singleton, we denote by $\argmin_{y\in C} f(y)$ (resp. $\argmax_{y\in C} f(y)$), with lower case ``a", the unique element of that set.

	\item Let $f$ be a function defined on a subset of $\R^d$, with values in $\R^p$ for some $p\geq 1$ (for us, in practice, $p=1$ or $d$). Then, given a point $x$ in the interior of the domain of $f$, we say that $f$ has a directional derivative at $x$ in the direction $t\in\R^d$ if and only if the quantity $\varepsilon^{-1}(f(x+\varepsilon t)-f(x))$ has a limit as $\varepsilon\to 0$, with $\varepsilon>0$. In that case, we denote this limit by $\diff^+f(x;t)$. Note that if $f$ has directional derivatives at $x\in\R^d$, then it must be continuous at $x$. Moreover, the map $\diff^+f(x;\cdot)$ is automatically measurable, since the limit can be taken along the sequence $\varepsilon=1/k, k\geq 1$. 
If the ratio $\varepsilon^{-1}(f(x+\varepsilon t)-f(x))$ converges uniformly in $t$ on all compact subsets of $\R^d$, we say that $f$ has directional derivatives at $x$ in Hadamard sense. This is equivalent to requiring that for all $t\in\R^d$, for all sequences $(t_n)_{n\geq 1}$ converging to $t$ and for all seuqences $(\varepsilon_n)_{n\geq 1}$ of positive numbers converging to $0$, $\varepsilon_n^{-1}(f(x+\varepsilon_n t_n)-f(x))$ has a (finite) limit as $n\to\infty$ (see, e.g., \cite[Chapter III]{fernholz1983mises}).
	
	\item If $f$ is differentiable at $x$, we denote by $\diff f(x;\cdot)$ its differential. That is, $\diff f(x;t)=\diff^+ f(x,t)=\nabla f(x)^\top t$ for all $t\in\R^d$.

	\item Given a convex set $G_0\subseteq\R^d$, when we talk about a convex function on $G_0$, we always mean that it takes finite values only, i.e., we only consider convex functions $f:G_0\to\R$, which may be the restriction to $G$ of some lower-semicontinuous convex function $\tilde f:\R^d\to\R\cup\{\infty\}$ whose domain contains $G_0$.

	\item We call \textit{random convex function} any map $f:\Omega\times G\to\R$, where $G\subseteq \R^d$ is some convex set, such that $f(\cdot,t)$ is measurable for all $t\in G$ and $f(\omega,\cdot)$ is convex for all $\omega\in \Omega$. We could only assume that $f(\omega,\cdot)$ is convex for $\PP$-almost all $\omega\in\Omega$, but this does not bring significantly more generality. Unless we need to emphasize the dependence on $\omega$ explicitly, we rather write $f(t)$ instead of $f(\omega,t)$ for simplicity.

	\item The covariance matrix of a random vector $X$ in $\R^d$ with two moments is defined as $\var(X)=\E[XX^\top]-\E[X]\E[X]^\top=\E[(X-\E[X])(X-\E[X])^\top]$. That is, for all vectors $u,v\in\R^d$, $u^\top \var(X)v=\cov(u^\top X,v^\top X)$. When $S\in\R^{d\times d}$ is symmetric, positive definite, we denote by $\var_S(X)=S\var(X)S=\var(SX)$ so that for all vectors $u,v\in\R^d$, we have the identity $u^\top\var_S(X)v=\cov(\langle u,X\rangle_S,\langle v,X\rangle_S)$. This is the matrix representation of the covariance operator of $X$ corresponding to the Euclidean structure defined by $S$. 

	\item For all vectors $u\in\R^d$ and symmetric, positive semi-definite matrices $V\in\R^{d\times d}$, we denote by $\mathcal N_d(u,V)$ the $d$-variate Gaussian distribution with mean $u$ and covariance matrix $V$.
	
\end{enumerate}

\section{Key lemmas about deterministic and random convex functions} \label{Sec:lemmas}

\subsection{On the behavior of convex functions and sequences of convex functions} \label{sec:lemmas1}

First, we state a minimum principle for convex functions, which we will use a few times in the next sections.

\begin{lemma} \label{lem:convexmin}
	Let $G_0\subset\R^d$ be an open convex set and $G\subseteq G_0$ be a closed convex subset. Let $f:G_0\to\R$ be a convex function and $K\subseteq G_0$ be any compact, convex set. If $\min_{t\in \partial K\cap G}f(t)>f(t_0)$ for some $t_0\in K\cap G$, then $\Argmin\limits_{t\in G} f(t)\subseteq K$ and it is not empty. 
\end{lemma}

\begin{remark}
\begin{itemize}
	\item Recall that a convex function defined on an open convex set is automatically continuous on that set \cite[Theorem 10.1]{rockafellar1970convex}, hence, it automatically reaches its bounds on any compact set. 
	\item The phrasing of this lemma is a bit technical, but a simpler version, when $G=G_0=\R^d$, says that if $f$ has one value inside $K$ that is smaller than all values taken on $\partial K$, then, it has at least one minimizer, and they all lie in $K$. We need this slightly more technical statement in order to deal with constrained $M$-estimation later.
\end{itemize}
\end{remark}

\begin{proof}
	Fix some arbitrary $t\in G\setminus K$ and let us show that necessarily, $f(t)>f(t_0)$. Set $\phi:\lambda\in [0,1]\mapsto f(t_0+\lambda(t-t_0))$, which is a convex function. First, note that $t_0\notin \partial K$ (or else, $t_0$ would be in $\partial K\cap G$ so $f(t_0)\geq \min_{\partial K\cap G}f$, which would contradict the assumption). Hence, there must be some $\lambda^*\in (0,1)$ such that $t_0+\lambda^*(t-t_0)\in\partial K$. Moreover, since both $t_0$ and $t$ are in $G$, $t_0+\lambda^*(t-t_0) \in G$.  Therefore, by assumption, $\phi(\lambda^*)>\phi(0)$. Hence, convexity of $\phi$ implies that it must be increasing on $[\lambda^*,1]$, yielding that $\phi(1)\geq \phi(\lambda^*)$ and hence, that $\phi(1)>\phi(0)$. That is, $f(t)>f(t_0)$.

Therefore, the minimizers (if any) of $f$ on $G$ must be contained in $K$. Finally, there must be at least one such minimizer since $f$ is continuous on the compact set $K\cap G$. 
\end{proof}


In the main statistical results presented in the next sections, Lemma~\ref{lem:convexmin} will be used to localize empirical minimizers of $\Phi_n$. 

The second key result is due to Rockafellar and shows that, for sequences of convex functions, uniform convergence can be deduced from pointwise convergence on a dense subset. From this lemma, we will derive two probabilistic corollaries. 

\begin{lemma}\cite[Theorem 10.8]{rockafellar1970convex} \label{lemma:Rockafellar}
	Let $G_0\subseteq \R^d$ be an open convex set and $f,f_1,f_2,\ldots$ be convex functions on $G_0$. Assume that there is a dense subset $C$ of $G_0$ such that for all $t\in C$, $f_n(t)\to f(t)$. Then, $f_n$ converges uniformly to $f$ on all compact subsets of $G_0$. 
\end{lemma}

An important consequence that we will use extensively is the following corollary.

\begin{corollary} \label{cor:Rockafellar_as}
	Let $f,f_1,f_2,\ldots$ be random convex functions defined on an open convex set $G_0\subseteq \R^d$. Assume that $f_n(t)\xrightarrow[n\to\infty]{} f(t)$ almost surely (resp. in probability) for all $t\in G_0$. Then, for all compact sets $K\subseteq G_0$, $\sup_K|f_n-f|\xrightarrow[n\to\infty]{} 0$ almost surely (resp. in probability). 
\end{corollary}

\begin{proof}
Let us prove the statement for the almost sure convergence and the convergence in probability separately.
\paragraph{\underline{Almost sure convergence}.} \,

Let $C$ be a dense and countable subset of $G_0$. By assumption, for each $t\in C$, it holds with probability one that $f_n(t)\xrightarrow[n\to\infty]{} f(t)$. Since $C$ is countable, this implies that with probability $1$, $f_n(t)\xrightarrow[n\to\infty]{} f(t)$ for all $t\in C$ simultaneously. Hence, by Lemma~\ref{lemma:Rockafellar}, with probability $1$, $f_n$ converges uniformly to $f$ on all compact subsets of $G_0$.

\paragraph{\underline{Convergence in probability}.} \,

Again, let $C$ be a dense and countable subset of $G_0$ and fix a compact subset $K$ of $G_0$. Our goal is to show that $Z_n:=\sup_{t\in K}|f_n(t)-f(t)|\xrightarrow[n\to\infty]{}0$ in probability. It is necessary and sufficient to show that every subsequence of $(Z_n)_{n\geq 1}$ has a further subsequence that converges to $0$ almost surely \cite[Section 3.3, Lemma 2]{chow2003probability}. With no loss of generality (since we could just renumber the terms of the sequence), let us prove that $(Z_n)_{n\geq 1}$ has a subsequence that converges to $0$ almost surely. Denote by $t_1,t_2,\ldots$ the elements of $C$. 

By assumption, $f_n(t_1)\xrightarrow[n\to\infty]{}f(t_1)$ in probability, so it has a subsequence that converges almost surely. That is, there is an increasing map $\psi_1:\N^*\to\N^*$ such that $f_{\psi_1(n)}(t_1)\xrightarrow[n\to\infty]{} f(t_1)$ almost surely. 

Similarly, $(f_{\psi_1(n)}(t_2))_{n\geq 1}$ being a subsequence of $(f_n(t_2))_{n\geq 1}$, it converges almost surely to $f(t_2)$ and thus has a further subsequence $(f_{\psi_1(\psi_2(n))}(t_2))_{n\geq 1}$ that converges almost surely to $f(t_2)$. By induction, one can construct a sequence of increasing maps $\psi_p:\N^*\to\N^*, p\geq 1$, such that for all integers $p\geq 1$, $f_{\psi_1\circ\ldots\circ\psi_p(n)}(t_p)$ converges to $f(t_p)$ almost surely. Let $\psi(n)=\psi_1\circ\ldots\circ\psi_n(n)$, for all $n\geq 1$. This is an increasing map; Let us prove that $Z_{\psi(n)}\xrightarrow[n\to\infty]{} 0$ almost surely, which will prove the lemma. 

First, note that with probablity $1$, $f_{\psi_1\circ\ldots\circ\psi_p(n)}(t_p)$ converges to $f(t_p)$ simultaneously for all $p\geq 1$. Second, for all $p\geq 1$, $(f_{\psi(n)}(t_p))_{n\geq 1}$ is a subsequence of $(f_{\psi_1\circ\ldots\circ\psi_p(n)}(t_p))_{n\geq 1}$ (except maybe for the first $p$ terms of the sequence).
Hence, $f_{\psi(n)}(t_p)\xrightarrow[n\to\infty]{} f(t_p)$ for all $p\geq 1$, almost surely. The rest follows from the first part of the proof (the case of almost sure convergence).  
\end{proof}

In fact, we can also derive a similar corollary for $L^p$ convergence, for any $p\geq 1$. We defer it to the appendix (Section~\ref{sec:misc}), because we only use it to formulate an open question, see the end of Section~\ref{sec:diff_case}).

\subsection{On the existence of measurable minimizers and measurable subgradients} \label{Sec:existence}

The existence of minimizers of a random convex function can often be established quite easily (for instance, if the function is coercive). Same for subgradients since any convex function defined on an open convex set has at least one subgradient at any point of that set. However, the existence of a measurable minimizer or subgradient is much less trivial and relies on the theory of measurable selections. 

\subsubsection{Measurable selections}

\begin{definition}
	Let $\Gamma:\Omega\to \mathcal P(\R^d)$ be a multifunction, that is, a function that maps any $\omega\in\Omega$ to some non-empty set $\Gamma(\omega)\subseteq\R^d$. A measurable selection of $\Gamma$ is a measurable map $\gamma:\Omega\to\R^d$ such that for all $\omega\in\Omega$, $\gamma(\omega)\in\Gamma(\omega)$. 
\end{definition}

There are numerous theorems that guarantee the existence of measurable selections in various setups, see \cite{himmelberg1982measurable,molchanov2005theory}. The one that we will need is the following, that follows from combining Theorems 3.2 (ii), 3.5 and 5.1 of \cite{himmelberg1982measurable}. Denote by $\mathcal C$ the collection of all non-empty, closed subsets of $\R^d$. 

\begin{lemma} \label{lem:measurable_selections}
	Let $\Gamma:\Omega\to\mathcal C$ be a multifunction. Assume that for all compact sets $K\subseteq \R^d$, the set $\{\omega\in\Omega:\Gamma(\omega)\cap K\neq \emptyset\}$ is measurable (that is, it belongs to the $\sigma$-algebra $\mathcal F$). Then, $\Gamma$ has a measurable selection. 
\end{lemma}

A multifunction satisfying this property above is called \textit{$C$-measurable} ($C$ as in ``compact", the test sets $K$ used in Lemma~\ref{lem:measurable_selections} being compact). 

\subsubsection{Measurable empirical risk minimizers} \,

From Lemma~\ref{lem:measurable_selections}, we obtain the following result, which will guarantee the existence of a measurable empirical risk minimizer for large enough $n$, and which will, at the same time, yield its strong consistency. 

\begin{theorem} \label{thm:measurable_min}
	Let $f,f_1,f_2,\ldots$ be random convex functions defined on an open convex set $G_0\subseteq \R^d$ such that for all $t\in G_0$, $f_n(t)\xrightarrow[n\to\infty]{} f(t)$ almost surely. Let $G\subseteq G_0$ be a closed, convex set. Assume that $G^*:=\Argmin_{t\in G}f(t)$ is non-empty and compact. Then, there exists a sequence $(t_n)_{n\geq 1}$ of random variables with values in $G$ such that with probability $1$, $t_n$ is a minimizer of $f_n$ on $G$ for all large enough $n$. Moreover, $\dist(t_n,G^*)\xrightarrow[n\to\infty]{}0$ almost surely.
\end{theorem}

\begin{proof}
For $n\geq 1$, let $M_n:=\Argmin_{t\in G} f_n(t)$, possibly empty. We proceed in two steps. First, we prove that with probability $1$, $M_n$ is non-empty for all large enough $n$. Second, we use the measurable selection to obtain such a sequence $(t_n)_{n\geq 1}$. 

\paragraph{\underline{Step 1}.} Note that if $G$ is compact, then $M_n\neq\emptyset$ for all $n\geq 1$, since $f_n$ is convex, hence continuous, on the open set $G_0$. 

First, Corollary~\ref{cor:Rockafellar_as} yields that $f_n$ converges uniformly to $f$ on any compact subset of $G_0$, almost surely. Fix some arbitrary, small enough $\varepsilon>0$ such that $G_\varepsilon^*:=\{t\in \R^d:\dist(t,G^*)\leq \varepsilon\}$. This set is compact, so
\begin{equation} \label{proof:2.1}
	\sup_{t\in G_\varepsilon^*\cap G} |f_n(t)-f(t)|\xrightarrow[n\to\infty]{}0.
\end{equation} 

Let $f^*:=\min_{t\in G}f(t)$ be the smallest value of $f$ on $G$ (note that $f^*$ is measurable, since it can be written as the infimum of $f(t)$ for $t$ ranging in a countable, dense subset of $G$). Convexity of $f$ on the open set $G_0$ implies its continuity. Therefore, $\eta:=\min_{t\in\partial G_\varepsilon^*\cap G} f(t)-f^*>0$. 

Then, the following holds with probability $1$: For all sufficiently large integers $n$ and for all $t\in\partial G_\varepsilon^* \cap G$, 
\begin{align*}
	f_n(t) & \geq f(t)-\eta/3 \quad \mbox{ by } \eqref{proof:2.1} \\
	& \geq f^*+\eta-\eta/3 \quad \mbox{ by definition of } \eta \\
	& \geq f_n(t^*)-\eta/3+\eta-\eta/3 \quad \mbox{ again by } \eqref{proof:2.1}  \\
	& = f_n(t^*)+\eta/3 > f_n(t^*).
\end{align*}
Therefore, by Lemma~\ref{lem:convexmin}, it holds with probability $1$ that, for all large enough integers $n\geq 1$, \begin{equation} \label{eqn:consist}
	\emptyset \neq M_n\subseteq G_{\varepsilon}^*.  
\end{equation}

\paragraph{\underline{Step 2}.} Now, fix an arbitrary element $t_0\in G$. For all integers $n\geq 1$, let $\Gamma_n:=\begin{cases} M_n \mbox{ if } M_n\neq\emptyset \\ \{t_0\} \mbox{ otherwise}. \end{cases}$
Let us prove that $\Gamma_n$ has a measurable selection, for all $n\neq 1$. Since $M_n$ is always closed (by continuity of $f_n$), $\Gamma_n$ is always non-empty and closed, so by Lemma~\ref{lem:measurable_selections}, it is sufficient to check that for each $n\geq 1$, the multiset function $\Gamma_n:\Omega\to\mathcal C$ is $C$-measurable in order to guarantee the existence of a measurable selection.

Fix $n\geq 1$ and let $K\subseteq \R^d$ be any compact set and let us show that the set $\{\omega\in\Omega:\Gamma_n(\omega)\cap K\neq\emptyset\}$ is a measurable set. 

First, rewrite $\{\omega\in\Omega:\Gamma_n(\omega)\cap K\neq\emptyset\}= \{\omega\in\Omega:M_n(\omega)\cap K\neq\emptyset\}\cup \{\omega\in\Omega: M_n(\omega)=\emptyset, t_0\in K\}$. Since $f_n(\omega,\cdot)$\footnote{recall that above, we only wrote $f_n(t)$ instead of $f_n(\omega,t)$ for simplicity.} is continuous for every $\omega\in\Omega$, the first set in this union can be rewritten as $\{\omega\in\Omega:\inf_{t\in G}f_n(\omega,t)=\inf_{t\in K\cap G}f_n(\omega,t)\}$. Again, using continuity of $f_n(\omega,\cdot)$ for all $\omega\in\Omega$, we can rewrite $\inf_{t\in G}f_n(\omega,t)$ and $\inf_{t\in K\cap G}f_n(\omega,t)$ as $\inf_{t\in \tilde G_1}f_n(\omega,t)$ and $\inf_{t\in \tilde G_2}f_n(\omega,t)$ respectively, where $G_1$ and $G_2$ are dense, countable subsets of $G$ and $K\cap G$ respectively. Therefore, both $\inf_{t\in G}f_n(\omega,t)$ and $\inf_{t\in K\cap G}f_n(\omega,t)$ are measurable (as maps from $\Omega$ to $\R\cup\{-\infty\}$) and we obtain that $\{\omega\in\Omega:M_n(\omega)\cap K\neq\emptyset\}\in\mathcal F$.

Now, $\{\omega\in\Omega: M_n(\omega)=\emptyset, t_0\in K\}$ is empty if $t_0\notin K$, which is measurable. If $t_0\in K$, it reduces to the set $\{\omega\in\Omega: M_n(\omega)=\emptyset\}$, which can be decomposed as 
$$\{\omega\in\Omega: M_n(\omega)=\emptyset\}=\bigcap_{p\in\N^*}\bigcup_{q\geq p+1}\{\omega\in \Omega:\min_{t\in G\cap B(t_0,q)}f_n(\omega,t)<\min_{t\in G\cap B(t_0,q)}f_n(\omega,t)\}$$
which, therefore, is also measurable. 

Finally, Lemma~\ref{lem:measurable_selections} implies the existence of a sequence $(t_n)_{n\geq 1}$ of random variables such that for all $n\geq 1$, $t_n\in \Gamma_n$. Furthermore, by Step 1 of this proof, we also obtain that with probability $1$, $t_n\in M_n$ for all large enough $n$.

\paragraph{\underline{Step 3}.}

Finally, following the reasoning of Step 1, \eqref{eqn:consist} yields that for all $\varepsilon>0$, it holds, with probability $1$, that $\dist(t_n,G^*)\leq \varepsilon$ for all large enough $n$. That is, $\dist(t_n,G^*)\xrightarrow[n\to\infty]{}0$ almost surely.

\end{proof}


\subsubsection{Measurable subgradients} \,

Now, we apply Lemma~\ref{lem:measurable_selections} to show the existence of measurable subgradients for random convex functions. Recall that for a convex function $f$ defined on a convex set $G_0\subseteq\R^d$, a subgradient of $f$ at a point $t_0\in G_0$ is any vector $u\in\R^d$ such that 
$$f(t)\geq f(t_0)+u^\top(t-t_0), \quad \forall t\in G_0.$$
We denote by $\partial f(t_0)$ the collection of all subgradients of $f$ at $t_0$. If $t_0\in\mathrm{int}(G_0)$, then $\partial f(t_0)$ is non-empty, compact and convex by Lemma~\ref{lem:prop_subgrad}. In particular, if $G_0$ is open, then $f$ has subgradients at every point of $G_0$. Now, if $f$ is a random convex function, the existence of a measurable subgradient (i.e., that is chosen in a measurable way) at $t_0\in\mathrm{int}(G_0)$ is granted by the following theorem.

\begin{theorem}
	Let $f$ be a random convex function defined on a convex set $G_0\subseteq\R^d$ and let $t_0\in\mathrm{int}(G_0)$. Then, $f$ has a measurable subgradient at $t_0$.
\end{theorem}

\begin{proof}
	Let $\Gamma=\partial f(t_0)$ be the set of subgradients of $f$ at $t_0$ (that is, for all $\omega\in\Omega$, $\Gamma(\omega)=\partial \left(f(\omega,\cdot)\right)(t_0)$). Since $t_0\in\inter(G_0)$, $\Gamma$ only takes non-empty values. Moreover, by Lemma~\ref{lem:prop_subgrad}, it always takes closed values, so $\Gamma$ is a $\mathcal C$-valued multifunction. Hence, it is sufficient to check that it is $C$-measurable in order to apply Lemma~\ref{lem:measurable_selections}.

Let $K\subseteq\R^d$ be any arbitrary compact set. Lemma~\ref{lemma:subgradients_interior} yields that $\Gamma\cap K\neq\emptyset$ if and only if there exists $u\in K$ with the property that $\sup_{t\in B(t_0,\varepsilon)} \left(u^\top (t-t_0)-f(t)+f(t_0)\right)\leq 0$ where $\varepsilon>0$ is any small enough positive number satisfying that $B(t_0,\varepsilon)\subseteq \inter(G_0)$. Since $f$ is convex, it is continuous on $\inter(G)$ and, hence, on $B(t_0,\varepsilon)$. Let $C$ be a fixed dense, countable subset of $B(t_0,\varepsilon)$. Then, $\Gamma\cap K\neq\emptyset$ if and only if there exists $u\in K$ for which $\sup_{t\in C} \left(u^\top (t-t_0)-f(t)+f(t_0)\right)\leq 0$. Let $h(\omega,u)=\sup_{t\in C} \left(u^\top (t-t_0)-f(\omega,t)+f(\omega,t_0)\right)$, for all $\omega\in\Omega$ and $u\in\R^d$ (again, here, we emphasize the dependence on $\omega\in\Omega$ for clarity, even though it was omitted above). First, note that for all $u\in\R^d$, $h(\cdot,u)$ is measurable, as the supremum of a countable family of measurable functions. Second, for all $\omega\in\Omega$, the function $h(\omega,\cdot)$ is convex as the supremum of affine functions, and it only takes finite values: Indeed, $C\subseteq B(t_0,\varepsilon)$ is bounded and $f(\omega,\cdot)$ is continuous on $B(t_0,\varepsilon)$. Hence, $h(\omega,\cdot)$ is continuous on $\R^d$. Therefore, since $K$ is compact, $\Gamma(\omega)\cap K\neq \emptyset$ if and only if $\min_{u\in K}h(\omega,u)\leq 0$, if and only if $\inf_{u\in \tilde K}h(\omega,u)\leq 0$, where $\tilde K$ is a fixed, countable, dense subset of $K$. Therefore, we obtain
$\DS \{\omega\in\Omega:\Gamma(\omega)\cap K \neq\emptyset\} = \{\omega\in \Omega: \inf_{u\in \tilde K} h(\omega,u)\leq 0\}$ which is measurable, since $\inf_{u\in \tilde K} h(\cdot,u)$ is a measurable map.

\end{proof}

Finally, let us state an incredibly simple yet powerful result that shows that for convex functions, there is no need to apply any dominated convergence theorem in order to swap expectations and (sub-) gradients. It is very easy to check that if $f_1$ and $f_2$ are two convex functions on a convex set $G_0\subseteq\R^d$, then for all $t_0\in G_0$, $\partial f_1(t_0)+\partial f_2(t_0)\subseteq \partial (f_1+f_2)(t_0)$.\footnote{The other inclusion is also true if $G_0$ has non-empty interior but, perhaps surprisingly, requires a nontrivial argument.} The following lemma shows that this fact still holds for generalized sums of convex functions.

\begin{theorem} \label{thm:measurable_subgradient}
	Let $f$ be a random convex function defined on a convex set $G_0\subseteq \R^d$. For all $t\in\mathrm{int}(G_0)$, let $g(t)$ be a measurable subgradient of $f$ at $t$. Let $p\geq 1$ be a real number and assume that for all $t\in G_0$, $f(t)\in L^p(\PP)$ and denote by $F(t)=\E[f(t)]$. Then, $F$ is a convex function and for all $t\in G_0$, $g(t)\in L^p(\PP)$ and 
	$$\E[g(t)]\in \partial F(t).$$
\end{theorem}

\begin{proof}

Fix $t_0\in \inter(G_0)$ and let $g(t_0)$ be a measurable subgradient of $h$ at $t_0$ (the existence of which is guaranteed by Theorem~\ref{thm:measurable_subgradient}).  In order to check that $g(t_0)\in L^p(\PP)$, it is necessary and sufficient to check that each of its $d$ coordinates are in $L^p(\PP)$ or, equivalently, that for all $v\in\R^d$, $|g(t_0)^\top v|^p$ is integrable. Fix an arbitrary $v\in \R^d$ and let $\varepsilon>0$ be such that $t_0+\varepsilon v$ and $t_0-\varepsilon v$ are in $G_0$ (such an $\varepsilon$ exists because $t_0\in \inter(G_0)$). Then, by definition of subgradients, $g(t_0)^\top v\leq \varepsilon^{-1}(f(t_0+\varepsilon v)-f(t_0))$ and $-g(t_0)^\top v\leq \varepsilon^{-1}(f(t_0-\varepsilon v)-f(t_0))$. That is, 
$$|g(t_0)^\top v|\leq \max(\varepsilon^{-1}(f(t_0+\varepsilon v)-f(t_0)),\varepsilon^{-1}(f(t_0-\varepsilon v)-f(t_0))).$$
Since the right hand side is in $L^p(\PP)$ by assumption, so is $g(t_0)^\top v$. The vector $v$ was arbitrary, so we conclude that $g(t_0)\in L^p(\PP)$. 

Now, for the rest of the proof, simply note that, again, by definition of subgradients, 
$$f(t)\geq f(t_0)+g(t_0)^\top (t-t_0)$$
holds for all $t\in G_0$. Taking the expectation, which is linear, yields that 
$$F(t)\geq F(t_0)+\E[g(t_0)]^\top (t-t_0)$$
which concludes the proof. 

\end{proof}

\begin{remark}
\,
\begin{itemize}
	\item In fact, to obtain that $g(t_0)\in L^p(\PP)$, it would have been sufficient to assume that $f(t)\in L^p(\PP)$ for all $t\in B(t_0,\varepsilon)$, for any arbitrary, small enough $\varepsilon>0$.
	\item As a consequence of Theorem~\ref{thm:measurable_subgradient}, if $F$ is differentiable at $t_0\in\inter(G_0)$, then $\E[g(t_0)]$ does not depend on the choice of the measurable selection $g(t_0)$ and it is automatically equal to $\nabla F(t_0)$ (since $\nabla F(t_0)$ is the only subgradient of $F$ at $t_0$, in that case). 
	\item In fact, Lemma~\ref{lem:diff_subgrad2} shows that if $F$ is differentiable at some $t_0\in\inter(G_0)$, then $f$ is almost surely differentiable at $t_0$, so in that case, any measurable selection $g(t_0)$ must satisfy $g(t_0)=\nabla f(t_0)$ almost surely.  
	\item To the best of our knowledge, the converse inclusion to Theorem~\ref{thm:measurable_subgradient} is unknown: Can all subgradients of $F$ at $t_0$ be written as $\E[g(t_0)]$ for some measurable $g(t_0)\in\partial f(t_0)$?
\end{itemize}
\end{remark}

\section{Consistency} \label{Sec:consistency}

Consistency of empirical risk minimizers with a convex loss function is automatically granted in a strong sense, thanks to Lemma~\ref{lem:convexmin} which allows to localize the $M$-estimator, for large enough $n$, in an arbitrarily small neighborhood of the set of population minimizers with probability $1$. In what follows, we consider a sequence $(\hat\theta_n)_{n\geq 1}$ of random variables such that with probability $1$, for all large enough $n$, $\hat \theta_n$ is a minimizer of $\Phi_n$ on $\Theta$. Existence of such a sequence is granted by Theorem~\ref{thm:measurable_min}.

\begin{theorem} \label{thm:consist}
	Assume that $\Theta^*$ is compact and non-empty. Then, $\dist(\hat\theta_n,\Theta^*)\xrightarrow[n\to\infty]{}0$ almost surely, as $n\to\infty$. 
\end{theorem}

The proof of this theorem can be found in \cite{haberman1989concavity} (the only difference here being that we do not assume that $\Theta=\R^d$), and it is a direct consequence of Theorem~\ref{thm:measurable_min} above.

\begin{remark}
Theorem~\ref{thm:consist} shows that any empirical minimizer becomes, with probability $1$, arbitrarily close to the set of population minimizers $\Theta^*$. A converse statement is generally not true, that is, there can be elements of $\Theta^*$ that  may never be approached by any empirical minimizer. For instance, let $E=\R^d$, $\Theta=B(0,1)$ and $\phi(x,\theta)=x^\top\theta$. Furthermore, assume that $X_1$ has the standard normal distribution. Then, $\Phi(\theta)=\E[X]^\top\theta=0$ for all $\theta\in \Theta$, so $\Theta^*=\Theta$. However, $\Phi_n(\theta)=\bar X_n^\top \theta$, so with probability $1$, the empirical minimizer is unique, given by $\hat\theta_n=-\bar X_n/\|\bar X_n\|$. 
\end{remark}

\section{Asymptotic distribution} \label{Sec:asymptnorm}

In this section, we assume that $\Argmin_{\theta\in\Theta} \Phi(\theta)$ is a singleton and we denote by $\theta^*=\argmin_{\theta\in\Theta}\Phi(\theta)$. 

\subsection{Non-differentiable case}

We first study asymptotic properties of $\hat\theta_n$ without assuming differentiability of $\Phi$ at $\theta^*$. That is, $\partial\Phi(\theta^*)$ may not be not a singleton. 

The following useful property is fundamental in that case. Recall that for a non-empty convex subset $C\subseteq\R^d$, we denote by $h_C:\R^d\to\R\cup\{\infty\}$ its support function.

\begin{proposition} \label{prop:non-smooth}
	Assume that $\phi(\cdot,\theta)\in L^2(P)$ for all $\theta\in\Theta_0$. Let $(\rho_n)_{n\geq 1}$ be any non-decreasing sequence of positive numbers diverging to $\infty$ as $n\to\infty$. Then, for all $\theta\in\Theta_0$ and  $t\in\R^d$,
$$\rho_n(\Phi_n(\theta+t/\rho_n)-\Phi_n(\theta))\xrightarrow[n\to\infty]{}h_{\partial\Phi(\theta)}(t)$$
in probability.
\end{proposition}

\begin{proof}

Fix $\theta\in\Theta_0$. For all $t\in\R^d$, define 
\begin{align*}
F_n(t) & = \rho_n\left(\Phi_n(\theta+t/\rho_n)-\Phi_n(\theta)-\frac{1}{n\rho_n}t^\top \sum_{i=1}^n g(X_i,\theta)\right) \\
& \quad\quad\quad -\rho_n\left(\Phi(\theta+t/\rho_n)-\Phi(\theta)-\frac{1}{\rho_n}t^\top \E[g(X_1,\theta)]\right).
\end{align*}
Write $F_n(t)=\sum_{i=1}^n(Z_{i,n}-\E[Z_{i,n}])$
where $Z_{i,n}=\frac{\rho_n}{n}(\phi(X_i,\theta+t/\rho_n)-\phi(X_i,\theta)-(1/\rho_n)t^\top g(X_i,\theta))$, for all $i=1,\ldots,n$. 
Convexity of $\phi(X_i,\cdot)$ yields that $0\leq Z_{i,n}\leq \frac{1}{n}t^\top (g(X_i,\theta+t/\rho_n)-g(X_i,\theta))$, for all $i=1,\ldots,n$. By Theorem~\ref{thm:measurable_subgradient}, each $Z_{i,n}$, $i=1,\ldots,n$, is square-integrable. Hence, taking the square and the expectation in the last display, 
$$\E[Z_{i,n}^2]\leq \frac{1}{n^2}\E[Y_n^2]$$
where $Y_n=t^\top (g(X_1,\theta+t/\rho_n)-g(X_1,\theta))$. Since $(\rho_n)_{n\geq 1}$ is non-decreasing, Lemma~\ref{lem:monoton_subgrad} implies that the sequence $(Y_n)_{n\geq 1}$ is non-increasing, yielding that $\E[Z_{i,n}^2]\leq \frac{1}{n^2}\E[Y_1^2]$ and, by independence of $X_1,X_2,\ldots$,
\begin{equation*}
	\var\left(\sum_{i=1}^n Z_{i,n}\right) = \sum_{i=1}^n \var(Z_{i,n})\leq \sum_{i=1}^n \E[Z_{i,n}^2] \leq \frac{\E[Y_1^2]}{n} \xrightarrow[n\to\infty]{} 0.
\end{equation*}
We conclude that $F_n(t)\xrightarrow[n\to\infty]{}0$ in $L^2$ and, hence, in probability. Now, rewrite $F_n(t)$ as
\begin{align}
	F_n(t) & = \rho_n(\Phi_n(\theta+t/\rho_n)-\Phi_n(\theta)) \nonumber \\
	& \quad\quad - t^\top\left(\frac{1}{n}\sum_{i=1}^n g(X_i,\theta)-\E[g(X_1,\theta)]\right) \label{term:1} \\
	& \quad\quad -\rho_n\left(\Phi(\theta+t/\rho_n)-\Phi(\theta)\right). \label{term:2}
\end{align}
The law of large numbers yields that the term \eqref{term:1} converges to $0$ in probability, and the term in \eqref{term:2} goes to $\diff^+\Phi(\theta;t)$ as $n\to\infty$. The result then follows from Lemma~\ref{lem:dirder_support}. 
\end{proof}

As a consequence, we obtain the following theorem. 

\begin{theorem} \label{thm:non-diff}
	Assume that $\phi(\cdot,\theta)\in L^2(P)$ for all $\theta\in\Theta_0$ and that  $0\in\inter(\partial\Phi(\theta^*))$. Then, $\hat\theta_n=\theta^*$ with probability going to $1$ as $n\to\infty$. 
\end{theorem}

Note that the assumption that $0\in\inter(\partial\Phi(\theta^*))$ readily implies that $\theta^*$ must be the unique minimizer of $\phi$ on $\Theta$ and even on $\Theta_0$. It also implies that $\Phi$ is not differentiable at $\theta^*$.

\begin{proof}
Let $(\rho_n)_{n\geq 1}$ be any non-decreasing sequence of positive numbers diverging to $\infty$ as $n\to\infty$. Since $\Theta_0$ is open, we can find $r>0$ such that $B(\theta^*,r)\subseteq \Theta_0$. For all $n\geq 1$, denote by $T_n=\{t\in\R^d:\theta^*+t/\rho_n\in \Theta\}=\rho_n(\Theta-\theta^*)$.  Finally, set $G_n(t)=\rho_n(\Phi_n(\theta^*+t/\rho_n)-\Phi_n(\theta^*))$, for all $t\in\R^d$ such that $\theta^*+t/\rho_n\in\Theta_0$. By definition of $\hat\theta_n$, $\hat t_n:=\rho_n(\hat\theta_n-\theta^*)$ is a minimizer of $G_n$ on $T_n$ for all large enough $n$, with probability $1$. 

Now, fix $\varepsilon>0$.
Combining Proposition~\ref{prop:non-smooth}, Corollary~\ref{cor:Rockafellar_as} and Lemma~\ref{lem:dirder_support}, we get 
$$\sup_{t\in B(0,\varepsilon)} |G_n(t)-h_{\partial \Phi(\theta^*)}(t)| \xrightarrow[n\to\infty]{} 0$$
in probability (note that $B(0,\varepsilon)\subseteq \rho_n(\Theta_0-\theta^*)$ for all large enough integers $n$).
Now, since $0\in\inter(\partial\Phi(\theta^*))$, the quantity $\eta:=\min_{u\in\R^d:\|u\|=1}h_{\partial\Phi(\theta^*)}(u)$ is positive. 

Assume that $n$ is large enough so $\sup_{t\in B(0,\varepsilon)} |G_n(t)-h_{\partial \Phi(\theta^*)}(t)|\leq \varepsilon\eta/2$ with probability at least $1-\varepsilon$. When this inequality is satisfied, we get that, for all $t\in T_n$ with $\|t\|=\varepsilon$, 
\begin{align*}
	G_n(t) & \geq h_{\partial \Phi(\theta^*)}(t)-\varepsilon\eta/2 \\
	& = \varepsilon h_{\partial \Phi(\theta^*)}(t/\varepsilon) - \varepsilon\eta/2 \quad \mbox{ by positive homogeneity of } h_{\partial\Phi(\theta^*)} \\
	& \geq \varepsilon\eta-\varepsilon\eta/2 \quad \mbox{ by definition of } \eta \\
	& > \varepsilon\eta/2 \\
	& > 0 = G_n(0)
\end{align*}
yielding, thanks to Lemma~\ref{lem:convexmin}, that $\|\hat t_n\|$ cannot be larger than $\varepsilon$. Hence, we have shown that for all large enough $n$, it holds with probability at least $1-\varepsilon$ that $\|\rho_n(\hat\theta_n-\theta^*)\|\leq \varepsilon$. That is, $\rho_n(\hat\theta_n-\theta^*)\xrightarrow[n\to\infty]{}0$ in probability. Since this must hold for any positive, non-decreasing sequence $(\rho_n)_{n\geq 1}$ diverging to $\infty$ as $n\to\infty$, Lemma~\ref{lem:convergence_proba} implies the desired statement.

\end{proof}

%
%
%
%
%
%

Let $C$ be the support cone to $\Theta$ at $\theta^*$. Recall that the first order condition (Lemma~\ref{lem:FOC_gen}) yields that $C\subseteq h_{\partial\Phi(\theta^*)}^{-1}([0,\infty))$. The next result extends Theorem~\ref{thm:non-diff}.

\begin{theorem} \label{thm:First_order_asymp}
Assume that $\phi(\cdot,\theta)\in L^2(P)$ for all $\theta\in\Theta_0$ and that $h_{\partial\Phi(\theta^*)}(t)>0$ for all $t\in C\setminus\{0\}$. Then, with probability going to $1$ as $n\to\infty$, $\hat\theta_n=\theta^*$. 
\end{theorem}

The assumption of the theorem is that the two closed, convex cones $C$ and $\{t\in\R^d:h_{\partial\Phi(\theta^*)}(t)\leq 0\}$ have a trivial intersection. Note that, by the first order condition at $\theta^*$, this intersection must always be included in the boundary of $C$. In other words, the assumption of the theorem is that all (non-zero) vectors in $C$ are directions of strict, linear increase of the population risk $\Phi$.

\begin{proof}

A consequence of the assumption of the theorem is that for all $\varepsilon>0$, $\{t\in C:h_{\partial\Phi(\theta^*)}(t)\leq \varepsilon\}$ is compact. Indeed, it is closed, since $C$ is closed and $h_{\partial\Phi(\theta^*)}$ is continuous. Moreover, the set $\{t\in C:\|t\|=1\}$ is compact, so by continuity of $h_{\partial\Phi(\theta^*)}$, there is some $t_0\in C$ with $\|t_0\|=1$ satisfying, for all $t\in C\setminus\{0\}$, $h_{\partial\Phi(\theta^*)}(t) \geq \|t\|h_{\partial\Phi(\theta^*)}(t_0)$. The assumption of the theorem implies that $h_{\partial\Phi(\theta^*)}(t_0)> 0$. Finally, $\{t\in C:h_{\partial\Phi(\theta^*)}(t)\leq\varepsilon\}$ is bounded, since it is included in $B(0,\varepsilon/h_{\partial\Phi(\theta^*)}(t_0))$. 

Now, let $(\rho_n)_{n\geq 1}$ be an arbitrary non-decreasing sequence of positive numbers, diverging to $\infty$ as $n\to\infty$ and fix $\varepsilon>0$. Proposition~\ref{prop:non-smooth}, Corollary~\ref{cor:Rockafellar_as} and Lemma~\ref{lem:dirder_support}, yield that 
$\sup_{t\in C:h_{\partial\Phi(\theta^*)}(t)\leq\varepsilon}|G_n(t)-h_{\partial\Phi(\theta^*)}(t)|\xrightarrow[n\to\infty]{}0$ in probability, where we set $G_n(t)=\rho_n(\Phi_n(\theta^*+t/\rho_n)-\Phi_n(\theta^*))$ as in the proof of Theorem~\ref{thm:non-diff}.
Let $n$ be large enough so $\sup_{t\in C:h_{\partial\Phi(\theta^*)}(t)\leq\varepsilon}|G_n(t)-h_{\partial\Phi(\theta^*)}(t)|\leq \varepsilon/2$ with probability at least $1-\varepsilon$. Then, with probability at least $1-\varepsilon$, it holds simultaneously for all $t\in T_n=\rho_n(\Theta-\theta^*)$ with $h_{\partial\Phi(\theta^*)}(t)=\varepsilon$, that
\begin{equation*}
	G_n(t) \geq h_{\partial\Phi(\theta^*)}(t)-\varepsilon/2 = \varepsilon/2   >0 = G_n(0)
\end{equation*}
so, by Lemma~\ref{lem:convexmin}, any minimizer $\hat t_n$ of $G_n$ on $T_n$ satisfies $h_{\partial\Phi(\theta^*)}(\hat t_n)\leq \varepsilon$. In particular, we obtain, for all large enough $n$, that with probability at least $1-\varepsilon$, 
$$0\leq h_{\partial\Phi(\theta^*)}(\rho_n(\hat \theta_n-\theta^*))= \rho_n h_{\partial\Phi(\theta^*)}(\hat \theta_n-\theta^*)\leq \varepsilon$$
where the first inequality follows from the first order condition for $\Phi$ at $\theta^*$ (Lemma~\ref{lem:FOC_gen}). That is $\rho_n h_{\partial\Phi(\theta^*)}(\hat \theta_n-\theta^*)\xrightarrow[n\to\infty]{}0$. Since the sequence $(\rho_n)_{n\geq 1}$ was arbitrary, Lemma~\ref{lem:convergence_proba} yields that $h_{\partial\Phi(\theta^*)}(\hat \theta_n-\theta^*)=0$ with probability going to $1$ as $n\to\infty$. Since $\hat\theta_n-\theta^*\in C$, this means that $ \hat \theta_n-\theta^*=0$ with probability going to $1$ as $n\to\infty$, which is the desired statement. 
\end{proof}

\begin{remark}
	Results of this section rely on Proposition~\ref{prop:non-smooth}, which imposes square-integrability of the loss function. We do not know whether the same results could be proved under weaker assumptions.
\end{remark}

Now, to obtain a more precise asymptotic description of $\hat\theta_n$ when $\Phi$ is differentiable at $\theta^*$ (this could be the case in Theorem~\ref{thm:First_order_asymp}, with $\nabla\Phi(\theta^*)^\top t>0$ for all $t\in C\setminus\{0\}$, but not in Theorem~\ref{thm:non-diff}), we will assume the existence of second order derivatives for $\Phi$ at $\theta^*$. This is the object of the next section. 

\subsection{Differentiable case} \label{sec:diff_case}

Let us first state the main result of this section. 

\begin{theorem} \label{thm:asympt_norm}
	Let $g:E\times\Theta_0\to\R^d$ be a measurable selection of subgradients of $\phi$. 
	Assume the following:
\begin{itemize}
	\item[(i)] $\Phi$ is twice differentiable at $\theta^*$ and $S:=\nabla^2\Phi(\theta^*)$ is positive definite;
	\item[(ii)] $g(\cdot,\theta^*)\in L^2(P)$;
	\item[(iii)] $\pi_{\Theta-\theta^*}^S$ has directional derivatives at $-S^{-1}\nabla\Phi(\theta^*)$.
\end{itemize}

Then, 
$$\sqrt n(\hat\theta_n-\theta^*)\xrightarrow[n\to\infty]{} \diff^+\pi_{\Theta-\theta^*}^{S}(-S^{-1}\nabla\Phi(\theta^*);Z)$$
in distribution, where $Z\sim \mathcal N_d(0,S^{-1}BS^{-1})$ and $B=\var(g(X_1,\theta^*))$.
\end{theorem}

\begin{remark}[on the assumptions of the theorem]
\, 
\begin{itemize}

	\item[(i)] Second differentiability of $\Phi$ at $\theta^*$ is not a strong restriction, since all convex functions are twice differentiable almost eveywhere in the interior of their domains \cite{aleksandorov1939almost}.  The assumption that $\nabla^2\Phi(\theta^*)$ is definite positive is made in order to obtain $n^{-1/2}$ convergence rate. This assumption could be relaxed, yielding slower rates under further, technical assumptions on higher order derivatives on $\Phi$. In this work, we choose to focus on the $n^{-1/2}$ rate because it only requires minimal, easy to check, non-restrictive smoothness assumptions. 
	
	\item[(ii)] Existence of the map $g$ is guaranteed by Theorem~\ref{thm:measurable_subgradient}. Moreover, the first assumption on $\Phi$ implies that it is differentiable at $\theta^*$, so by Lemma~\ref{lem:diff_subgrad2}, $\phi(X_1,\cdot)$ is almost surely differentiable at $\theta^*$ yielding that $g(x,\theta^*)=\nabla\left(\phi(x,\cdot)\right)(\theta^*)$ for $P$-almost all $x\in E$.  Theorem~\ref{thm:measurable_subgradient} also ensures that it is sufficient that $\phi(\cdot,\theta)\in L^2(P)$ for all $\theta\in \Theta_0$ for the second assumption to hold. In fact, a straightforward adaptation of Theorem~\ref{thm:measurable_subgradient} shows that it is even enough to only assume that $\phi(\cdot,\theta)\in L^2(P)$ for all $\theta$ in any arbitrarily small neighborhood of $\theta^*$. Note that this does not require a uniform domination of $\phi$ or its derivatives/subgradients in any neighborhood of $\theta^*$ but, rather, a pointwise integrability condition of order $0$ (that is, on $\phi$ itself).

	\item[(iii-a)] Directional differentiability of $\pi_{\Theta-\theta^*}^S$ is not a strong restriction in the sense that, $\pi_{\Theta-\theta^*}^S$ being non-expansive (see Lemma~\ref{lem:charact_proj}) it is automatically differentiable almost everywhere by Rademacher's theorem \cite[Section 3.1.6, p. 216]{federer1969geometric}. In the appendix (Section~\ref{sec:appendix_diff_proj}), we present several sufficient conditions that guarantee the existence of directional derivatives of $\pi_K^S$ for a convex set $K$, at a direction $u$, which, in practice, are easily checked (e.g., $u\in K$, or $u\notin K$ and $\partial K$ is smooth at $\pi_K(u)$, or $K$ is defined by finitely many linear convex constraints, etc.). By an obvious linear change of variables, it is clear that the existence of a directional derivative of $\pi_{\Theta-\theta^*}^S$ at $-S^{-1}\nabla\Phi(\theta^*)$ in a direction $z\in\R^d$ is equivalent to the existence of a directional derivative of $\pi_{S^{1/2}(\Theta-\theta^*)}$ at $-S^{-1/2}\nabla\Phi(\theta^*)$ in the direction $S^{1/2}z$. Then, simple algebra yields that 
$$\diff^+\pi_{\Theta-\theta^*}^{S}(-S^{-1}\nabla\Phi(\theta^*);z)=S^{-1/2}\diff^+\pi_{S^{1/2}(\Theta-\theta^*)}(-S^{-1/2}\nabla\Phi(\theta^*);S^{1/2}z).$$
Recall that $(\theta-\theta^*)^\top \nabla\Phi(\theta^*)\geq 0$ for all $\theta\in\Theta$: This is granted by the first order condition at $\theta^*$ (Lemma~\ref{lem:FOC_gen}). That is, $-\nabla\Phi(\theta^*)$ is in the normal cone to $\Theta$ at $\theta^*$ or, equivalently, $-S^{-1/2}\nabla\Phi(\theta^*)$ is in the normal cone to $S^{1/2}(\Theta-\theta^*)$ at $0$. 
\end{itemize}
\end{remark}

\begin{remark}[on the conclusion of the theorem]
\,
\begin{itemize}
	 \item Lemma~\ref{lem:diff_proj_sanity} yields that for any $z\in\R^d$, $\diff^+\pi_{\Theta-\theta^*}^S(-S^{-1}\nabla\Phi(\theta^*);z)\in C_{S^{-1}\nabla\Phi(\theta^*)}^S=C_{\nabla\Phi(\theta^*)}$ where $C$ is the support cone to $\Theta$ at $\theta^*$. When $\nabla\Phi(\theta^*)^\top t>0$ for all $t\in C\setminus\{0\}$ (that is, $-\nabla\Phi(\theta^*)$ is in the interior of the normal cone to $\Theta$ at $\theta^*$), $C_{\nabla\Phi(\theta^*)}=\{0\}$, $\diff^+\pi_{\Theta-\theta^*}^S(-S^{-1}\nabla\Phi(\theta^*);\cdot)=0$ so Theorem~\ref{thm:asympt_norm} yields that $\sqrt n(\hat\theta_n-\theta^*)\xrightarrow[n\to\infty]{}0$ in distribution: This was already a (rather weak) consequence of Theorem~\ref{thm:First_order_asymp}.

	\item If $\theta^*\in\inter(\Theta)$, then the first order condition (Lemma~\ref{lem:FOC_gen}) yields that $\nabla\Phi(\theta^*)=0$ and, $\diff^+\pi_{\Theta-\theta^*}^S(0;\cdot)$ is simply the identity map. Therefore, Theorem~\ref{thm:asympt_norm} says that $\sqrt n(\hat\theta_n-\theta^*)\xrightarrow[n\to\infty]{}Z$ in distribution. In that case, Theorem~\ref{thm:consist} implies that, with probability $1$, for all large enough $n$, $\hat\theta_n\in\inter(\Theta)$. Hence, with probability $1$, for all large enough $n$, $\hat\theta_n$ (the constrained $M$-estimator) is also a solution to the unconstrained optimization problem $\min_{\theta\in\Theta_0}\Phi_n(\theta)$, and we recover Haberman's theorem \cite[Theorem 6.1]{haberman1989concavity}.
	
	\item In fact, Theorem~\ref{thm:asympt_norm} also encompasses the unconstrained case, by taking $\Theta=\Theta_0=\R^d$. If $\Theta_0$ is a strict open subset of $\R^d$, one can also consider an unconstrained $M$-estimator $\tilde\theta_n$ on the open set $\Theta_0$, that is, a minimizer of $\Phi_n$ on $\Theta_0$. Assume that $\theta^*$ is the unique minimizer of $\Phi$ on the open set $\Theta_0$ and let $\Theta$ be any closed subset of $\Theta_0$ containing $\theta^*$ in its interior (e.g., take $\Theta=B(\theta^*,\varepsilon)$ for any small enough $\varepsilon$). Then, a straight adaptation of Theorem~\ref{thm:consist} yields that $\tilde\theta_n\xrightarrow[n\to\infty]{}\theta^*$ almost surely, so $\tilde\theta_n\in \Theta$ for all large enough $n$, with probability $1$. That is, $\tilde\theta_n$ eventually coincides with a constrained $M$-estimator and, hence, also satisfies the conclusion of Theorem~\ref{thm:asympt_norm}, with $\diff^+\pi_{\Theta-\theta^*}^S(0;\cdot)$ being the identity map (note that in the case $\Theta=\Theta_0=\R^d$, we necessarily have that $\nabla\Phi(\theta^*)=0$).
	
	\item If the boundary of $\Theta$ is $C^2$ in a neighborhood of $\theta^*$ (that is, it can be locally represented as the graph of a $C^2$ mapping from $\R^{d-1}$ to $\R$) and $\nabla\Phi(\theta^*)\neq 0$, then, Lemma~\ref{lem:diffprojCk} yields that $\sqrt n(\hat\theta_n-\theta^*)$ converges in distribution to a Gaussian distribution that is supported in the linear hyperplane that is parallel to the (unique) supporting hyperplane to $\Theta$ at $\theta^*$. 
	
		\item Lemmas~\ref{lem:monotonicity_dirder} and \ref{lem:monotonicity_dirder2} imply that for all $t,t'\geq 0$ with $t'>t$, 
\begin{equation} \label{eqn:misspec}
	\|\diff^+\pi_{\Theta-\theta^*}^S(-t'S^{-1}\nabla\Phi(\theta^*);Z)\|_S\leq \|\diff^+\pi_{\Theta-\theta^*}^S(-tS^{-1}\nabla\Phi(\theta^*);Z)\|_S
\end{equation}
almost surely. This can be interpreted as follows. First, note that the set $\Theta$ can represent some constraints that are imposed by a specific application, or it can represent a model (e.g., if it is believed that the global minimizer of $\Phi$ lies in $\Theta$). In the latter case, the model is misspecified if the global minimizer of $\Phi$ is not in $\Theta$, that is, if $\nabla\Phi(\theta^*)\neq 0$. In other words, the vector $\nabla\Phi(\theta^*)$ (or its rescaled version $S^{-1}\nabla\Phi(\theta^*)$ can be used to quantify the amount of model misspecification. In that regard, \eqref{eqn:misspec} suggests that more misspecification yields better asymptotic error (we do not account for any misspecification bias here). In \eqref{eqn:misspec}, $t=0$ can be thought of as corresponding to the well-specified case. This will be illustrated in the examples below.
	
	\item As a consequence of Theorem~\ref{thm:asympt_norm}, the mean squared error of $\hat\theta_n$ satisfies
\begin{equation} \label{eq:MSE}\liminf_{n\to\infty}n\E[\|\hat\theta_n-\theta^*\|_S^2]\geq \E[\|\diff^+\pi_{\Theta-\theta^*}^S(-S^{-1}\nabla\Phi(\theta^*);Z)\|_S^2]
\end{equation}
(we do not know, in general, whether this is in fact an equality, with the $\liminf$ being a simple limit, see the open question below). The right hand side can be interpreted as a local measure of the statistical complexity of $\Theta$ around $\theta^*$, relative to the (population) loss function $\Phi$. The statistical dimension (or Gaussian width) of a non-empty, closed, convex set $G\subseteq\R^d$ is measured as $\E[\|\pi_G(Z)\|^2]$ where $Z\sim\mathcal N_d(0,I_d)$, see \cite{amelunxen2014living} (in our case, we need to account for a scaling given by $S^{-1}$ and $B$ in the covariance matrix of $Z$). In \eqref{eq:MSE}, we do not have a projection, but the directional derivative of a projection. The right hand side of \eqref{eq:MSE} can rather be seen as a statistical dimension at an infinitesimal scale. We can refer, for instance, to \cite{chatterjee2014new} who studied least squares under convex constraint, and proved that the statistical dimension at a fixed scale drives the statistical error. A similar phenomenon has also been studied for constrained $M$-estimators in a more general setup \cite{li2015geometric}. Recall, however, that except in specific cases (see Section~\ref{sec:appendix_diff_proj} in the appendix), $\diff^+\pi_{\Theta-\theta^*}^S(-S^{-1}\nabla\Phi(\theta^*);\cdot)$ is not the projection onto a convex set. 

	\item It is worth mentioning some further important properties of $\Pi:=\diff^+\pi_{\Theta-\theta^*}^S(-S^{-1}\nabla\Phi(\theta^*);\cdot)$. As we have noted above, in general, it is not the projection onto a convex cone. Nevertheless, it shares similar properties as the projection onto a convex cone. Indeed, by Lemma~\ref{lem:diff_proj_prop}, it satisfies the following properties:
\begin{itemize}
	\item $\Pi(\lambda z)=\lambda\Pi(z)$, for all $\lambda\geq 0$ and $z\in\R^d$ (positive homogeneity);
	\item $\|\Pi(z')-\Pi(z)\|_S\leq \|z'-z\|_S^2$ (non-expansiveness);
	\item $\langle\Pi(z')-\Pi(z),z'-z\rangle_S\geq \|\Pi(z')-\Pi(z)\|_S^2\geq 0$ for all $z,z'\in\R^d$ (firm monotonicity). 
\end{itemize}
Note that non-expansiveness is implied by firm monotonicity. Such maps satisfying the last two properties above have been studied extensively \cite{ZARANTONELLO1971237}. Moreover, \cite[Proposition 2.1]{noll1995directional} implies that $\Pi$ is the gradient of a convex function. 

\end{itemize}

\end{remark}


Now, let us look at some applications of Theorem~\ref{thm:asympt_norm}.

\begin{example}[Constrained mean estimation]
	Let $X_1,X_2,\ldots$ be iid random vectors with two moments\footnote{In fact, one moment is enough if one rather uses the loss function $\phi(x,\theta)=\|x-\theta\|^2-\|x\|^2$, $x,\theta\in\R^d$} and $\Theta\subseteq \R^d$ be a non-empty, closed, convex set. Consider the loss function $\phi(x,\theta)=(1/2)\|x-\theta\|^2, x,\theta\in \R^d$. Then, $\theta^*=\pi_{\Theta}(\E[X_1])$ is the unique minimizer of $\Phi$ on $\Theta$ and $\hat\theta_n=\pi_{\Theta}(\bar X_n)$ where $\bar X_n=n^{-1}(X_1+\ldots+X_n)$, for all $n\geq 1$. Consistency, which is a consequence of Theorem~\ref{thm:consist}, also follows directly from the strong law of large numbers, together with continuity of $\pi_\Theta$ (since it is non-expansive). For asymptotic normality, we obtain, from Theorem~\ref{thm:asympt_norm}, that 
	$$\sqrt n(\hat\theta_n-\theta^*)\xrightarrow[n\to\infty]{} \diff^+\pi_{\Theta-\theta^*}(\E[X_1]-\theta^*;Z)=\diff^+\pi_{\Theta}(\E[X_1];Z)$$
in distribution, where $Z\sim\mathcal N_d(0,\var(X_1))$ (in this example, $S=I_d$).
In this simple case, this result can also be obtained using the central limit theorem, combined with the delta method.\footnote{Delta method requires Hadamard directional differentiability of $\pi_{\Theta-\theta^*}$ at $\E[X_1]-\theta^*$. This is readily implied by the existence of directional derivatives together with non-expansiveness of $\pi_{\Theta-\theta^*}$}

Here, it is clear that misspecification is favorable for the asymptotic error: For instance, if $\Theta-\theta^*$ is a convex cone and $\E[X_1]-\theta^*$ is in the interior of the normal cone to $\Theta$ at $\theta^*$ (in particular, $\theta^*\neq \E[X_1]$), then, Theorem~\ref{thm:non-diff} yields that $\hat\theta_n=\theta^*$ with probability going to $1$ as $n\to\infty$. 
\end{example}

\begin{example}[Constrained least squares]
Let $(X_1,Y_1),(X_2,Y_2),\ldots$ be iid random pairs in $\R^d\times\R$. Assume that $X_1$ has four moments, $\E[X_1]=0$, $S:=\E[X_1X_1^\top]$ is definite positive, $Y_1-X_1^\top\theta_0$ is independent of $X_1$ and has the centered Gaussian distribution with variance $\sigma^2>0$ for some $\theta_0\in\R^d$ and $\sigma^2>0$. Let $\phi(x,y,\theta)=1/2(y-x^\top\theta)^2$, for all $x\in\R^d, y\in\R$ and $\theta\in\R^d$. Then, for all $\theta\in\R^d$, 
$$\Phi(\theta)=\frac{1}{2}\|\theta-\theta_0\|_S^2+\sigma^2.$$
Let $\Theta\subseteq\R^d$ be a non-empty, closed, convex subset of $\R^d$ (here, $\Theta_0=\R^d$). Then, $\Argmin_{\theta\in\Theta}\Phi(\theta)=\{\pi_\Theta^S(\theta_0)\}$ and, provided that $\pi_\Theta$ has directional derivatives at $\theta_0$, the least square estimator $\hat\theta_n$, defined as any minimizer on $\Theta$ of $\Phi_n(\theta)=n^{-1}\sum_{i=1}^n (Y_i-X_i^\top\theta)^2, \theta\in\R^d$, satisfies 
$$\sqrt n(\hat\theta_n-\theta^*)\xrightarrow[n\to\infty]{}\diff^+\pi_{\Theta-\theta^*}^S(\theta_0-\theta^*;Z)=\diff^+\pi_{\Theta}^S(\theta_0;Z)$$
in distribution, where $Z\sim\mathcal N_d(0,S^{-1}BS^{-1})$ and 
\begin{align*}
	B & = \var((Y_1-X_1^\top\theta^*)X_1) \\
	& =\var((Y_1-X_1^\top\theta_0)X_1+X_1^\top(\theta^*-\theta_0)X_1) \\
	& = \E[(X_1^\top(\theta_0-\theta^*))^2X_1X_1^\top]+\sigma^2 S.
\end{align*}


\end{example}

\begin{example}[Geometric median]
	Let $X_1,X_2,\ldots$ be iid random vectors with one moment.\footnote{Similarly to the first example, one need not assume the existence of one moment if the loss function is replaced with $\phi(x,\theta)=\|x-\theta\|-\|x\|$, $x,\theta\in\R^d$.} Consider the loss function $\phi(x,\theta)=\|x-\theta\|$, $x,\theta\in\R^d$. Then, $\theta^*$ is any geometric median and $\hat\theta_n$ is any empirical geometric median. Here, in the unconstrained case, we recover standard results for geometric median $M$-estimation, provided that the distribution of $X_1$ is not supported on an affine line (this guarantees uniqueness of $\theta^*$) and that $1/\|X_1-\theta^*\|$ is integrable (this guarantees that $\Phi$ is twice differentiable at $\theta^*$ with positive definite Hessian), see, e.g., \cite{koltchinskii1994bahadur}.
\end{example}

\begin{proof}[Proof of Theorem~\ref{thm:asympt_norm}]

Recall that we denote by $S=\nabla^2\Phi(\theta^*)$, which is a symmetric, positive definite matrix, by assumption.

First, since $\Theta_0$ is open, there exists some $r>0$ such that $B_S(\theta^*,r)\subseteq \Theta_0$. Fix some $R>0$, whose value will be determined later, and let $n\geq 1$ be any integer that is large enough so $R/\sqrt n\leq r$. For all such integers $n$, let $F_n$ be the random function defined on $B_S(0,R)$ by 
\begin{equation} \label{eq:Fn}
F_n(t)=n\left(\Phi_n(\theta^*+t/\sqrt n)-\Phi_n(\theta^*)\right)-\left(\frac{t^\top}{\sqrt n}\sum_{i=1}^n g(X_i,\theta^*)+\frac{1}{2}t^\top \nabla^2\Phi(\theta^*)t\right)\end{equation}
for all $t\in B_S(0,R)$. This is a random convex function. Our first goal is to prove that $F_n$ converges pointwise (and hence, by Corollary~\ref{cor:Rockafellar_as}, uniformly on the compact set $B_S(0,R)$) to zero in probability. From this, we will then obtain that any minimizer of the first term (one of which is given by $\sqrt n(\hat\theta_n-\theta^*)$ for large enough $n$, with probability $1$) is close to the unique minimizer of the second, quadratic term.

Fix $t\in B_S(0,R)$ and $n\geq 1$. For $i=1,\ldots,n$, let $Z_{i,n}=\phi(X_i,\theta^*+n^{-1/2}t)-\phi(X_i,\theta^*)-n^{-1/2}t^\top g(X_i,\theta^*)$. 
By definition of subgradients, 
$$0\leq Z_{i,n}\leq n^{-1/2}t^\top (g(X_i,\theta^*+n^{-1/2}t)-g(X_i,\theta^*)).$$
Squaring and taking the expectation yields that 
\begin{equation} \label{proof:4.2}
	\E[Z_{i,n}^2]\leq n^{-1}\E\left[\left(t^\top (g(X_1,\theta^*+n^{-1/2}t)-g(X_1,\theta^*))\right)^2\right]
\end{equation}
(we replaced $i$ with $1$ in the right hand side because the $X_i$'s are iid). Let $Y_n:=t^\top (g(X_1,\theta^*+n^{-1/2}t)-g(X_1,\theta^*))$. As mentioned above, $Y_n\geq 0$. Moreover, for $n\geq 1$, letting $u=\theta^*+t/\sqrt n$ and $v=\theta^*+t/\sqrt{n+1}$,
\begin{align*}
	Y_n-Y_{n+1} & = t^\top \left(g(X_1,u)-g(X_1,v)\right) \\
	& = (1/\sqrt n-1/\sqrt{n+1})^{-1}(u-v)^\top \left(g(X_1,u)-g(X_1,v)\right) \\
	& \geq 0
\end{align*}
by Lemma~\ref{lem:monoton_subgrad}. So the sequence $(Y_n)_{n\geq 1}$ is non-increasing. Hence, $Y_n$ converges almost surely to some non-negative random variable $Y$. By monotone convergence (noting that $Y_1$ is integrable), this implies that 
\begin{equation}	 \label{proof:4.1}
\E[Y_n]\xrightarrow[n\to\infty]{} \E[Y].
\end{equation} 
However, for all $n\geq 1$, $\E[Y_n]=t^{\top}\left(w_n-\nabla\Phi(\theta^*)\right)$ where $w_n\in\partial \Phi(\theta^*+t/\sqrt n)$, by Lemma~\ref{lem:diff_subgrad}. Lemma~\ref{lem:seq_subgrad} yielding that $w_n\xrightarrow[n\to\infty]{} w$, we obtain that $\E[Y_n]\xrightarrow[n\to\infty]{} 0$. 
Together with \eqref{proof:4.1}, this shows that $\E[Y]=0$ and, hence, because $Y\geq 0$, that $Y=0$ almost surely. Therefore, again by monotone convergence (noting, this time, that $Y_1^2$ is iontegrable), $\E[Y_n^2]\xrightarrow[n\to\infty]{} \E[Y^2]=0$. 

Combined with \eqref{proof:4.2} and using independence of $Z_{1,n},\ldots,Z_{n,n}$, we obtain that 
\begin{equation} \label{eqn:convL2}\var\left(\sum_{i=1}^n Z_{i,n}\right) = \sum_{i=1}^n \var(Z_{i,n}) \leq \sum_{i=1}^n \E[Z_{i,n}^2] \leq \E[Y_n^2] \xrightarrow[n\to\infty]{} 0.
\end{equation} 
Therefore, by Chebychev's inequality, $\sum_{i=1}^n (Z_{i,n}-\E[Z_{i,n}])\xrightarrow[n\to\infty]{} 0$ in probability, that is, 
$$n\Big(\Phi_n(\theta^*+n^{-1/2}t)-\Phi_n(\theta^*)\Big)-n^{-1/2}t^\top \sum_{i=1}^n g(X_i,\theta^*)-n\Big(\Phi(\theta^*+n^{-1/2}t)-\Phi(\theta^*)-n^{-1/2}t^\top \nabla\Phi(\theta^*)\Big) \xrightarrow[n\to\infty]{} 0$$ 
in probability. Now, since we have assumed that $\Phi$ is twice differentiable at $\theta^*$, we finally obtain that 
\begin{equation} \label{eqn:pointwisePconv}
	F_n(t)\xrightarrow[n\to\infty]{} 0
\end{equation} 
in probability, for all $t\in B_S(0,R)$, as desired.

For all integers $n\geq 1$, let $T_n=\{t\in\R^d:\theta^*+n^{-1/2}t\in\Theta\}=n^{1/2}(\Theta-\theta^*)\subseteq T$ and $S_n=\{t\in\R^d:\theta^*+n^{-1/2}t\in\Theta_0\}=n^{1/2}(\Theta_0-\theta^*)$. Then, $T_n$ is a closed subset of $S_n$. Moreover, since $\theta^*\in\Theta_0$ and $\Theta_0$ is open, $B_S(0,R)\subseteq S_n$ for all large enough integers $n$ (recall that $R>0$ is some fixed number, whose value is still to be determined). Define the maps 
$$\hat G_n:t\in S_n\mapsto n\Big(\Phi_n(\theta^*+n^{-1/2}t)-\Phi_n(\theta^*)\Big)$$ and 
$$G_n : t\in\R^d \mapsto \mbox{ } n^{-1/2}t^\top\sum_{i=1}^n g(X_i,\theta^*)+\frac{1}{2}t^\top \nabla^2\Phi(\theta^*)t.$$

As per these definitions, $F_n=\hat G_n-G_n$, so, \eqref{eqn:pointwisePconv} and  Corollary~\ref{cor:Rockafellar_as} yield that 
\begin{equation} \label{eqn:Unif_Conv}
	\sup_{t\in B_S(0,R)}|\hat G_n(t)-G_n(t)|\xrightarrow[n\to\infty]{} 0
\end{equation}
in probability.

Moreover, $\hat t_n:=n^{1/2}(\hat \theta_n-\theta^*)$ is a minimizer of $\hat G_n$ on $T_n$, by definition of the empirical risk minimizer $\hat\theta_n$.

Now, denote by $Z_n=n^{-1/2}S^{-1}\sum_{i=1}^n g(X_i,\theta^*)-\nabla\Phi(\theta^*)$ and for all $t\in\R^d$, rewrite $G_n(t)$ as 

\begin{align*}
G_n(t) & = n^{-1/2}t^\top\sum_{i=1}^n g(X_i,\theta^*)+\frac{1}{2}t^\top \nabla^2\Phi(\theta^*)t \\
& = \langle n^{-1/2}S^{-1}\sum_{i=1}^n g(X_i,\theta^*),t\rangle_S+\frac{1}{2}\|t\|_S^2 \\
& = \langle Z_n+\sqrt n S^{-1}\nabla\Phi(\theta^*),t\rangle _S +\frac{1}{2}\|t\|_S^2 \\
& = \frac{1}{2}\|t+Z_n+\sqrt n S^{-1}\nabla\Phi(\theta^*)\|_S^2-\|Z_n+\sqrt n S^{-1}\nabla\Phi(\theta^*)\|_S^2.
\end{align*}

It is now clear that $G_n$ has a unique minimizer on $T_n$, which we denote by $t_n^*$ and which is given by 
$$t_n^*=\pi_{T_n}^S(-Z_n-\sqrt n S^{-1}\nabla\Phi(\theta^*)).$$

Now, our goal is twofold. First, to study the asymptotic behavior of $t_n^*$ and show that it converges in distribution, as $n\to\infty$. Second, to check, based on \eqref{eqn:Unif_Conv}, that $\hat t_n$ approaches $t_n^*$ as $n\to\infty$, that is, $\hat t_n-t_n^*$ converges in probability to $0$. Using Slutsky's theorem, these two facts will imply convergence in distribution of $\hat t_n$.

\paragraph{\underline{Asymptotic behavior of $t_n^*$}.} \,

First, by the central limit theorem, we have that $Z_n\xrightarrow[n\to\infty]{} Z$ in distribution, where $Z$ is is a centered Gaussian random variable with covariance matrix given by $S^{-1}\var(g(X_1,\theta^*))S^{-1}$. 

By Skorohod representation theorem (see \cite[Theorem 5.31]{kallenberg1997foundations} for instance), one may assume that $Z_n$ converges almost surely to $Z$. Since $\pi_C^S$ is non-expansive by Lemma~\ref{lem:charact_proj}, it holds that $t_n^*-\pi_{T_n}^S(-Z-\sqrt nS^{-1}\nabla\Phi(\theta^*))$ converges to $0$ almost surely. Moreover, 
\begin{align*}
	\pi_{T_n}^S(-Z-\sqrt nS^{-1}\nabla\Phi(\theta^*)) & = \pi_{\sqrt n(\Theta-\theta^*)}^S (-Z-\sqrt nS^{-1}\nabla\Phi(\theta^*)) \\
	& = \sqrt n \pi_{\Theta-\theta^*}^S(-n^{-1/2}Z-S^{-1}\nabla\Phi(\theta^*)) \\
	& \xrightarrow[n\to\infty]{} \diff^+\pi_{\Theta-\theta^*}^S(-S^{-1}\nabla\Phi(\theta^*);-Z) 
\end{align*}
almost surely, using the third assumption of the theorem.
Therefore, we conclude that $t_n^*\xrightarrow[n\to\infty]{} \diff^+\pi_{\Theta-\theta^*}^S(-S^{-1}\nabla\Phi(\theta^*);-Z)$ almost surely and, hence, in distribution. The desired results follows, since $Z$ and $-Z$ are identically distributed. 

\paragraph{\underline{Convergence in probability of $\hat t_n-t_n^*$ to $0$}.} \,

Fix $\varepsilon>0$. Since the sequence $(t_n^*)_{n\geq 1}$ converges in distribution (see the previous paragraph), it is tight, that is, there must exist some $M>0$ such that for all $n\geq 1$, $P(\|t_n^*\|_S\leq M)\geq 1-\varepsilon$. 
Let $K=B_S(0,M+\varepsilon)$ and fix some $\eta>0$ to be chosen below. \eqref{eqn:Unif_Conv} yields that for all large enough $n\geq 1$, $\sup_{t\in K}|\hat G_n(t)-G_n(t)|\leq\eta$ with probability at least $1-\varepsilon$. Therefore, by the union bound, for all large enough $n\geq 1$, it holds with probability at least $1-2\varepsilon$ that simultaneously for all $t\in T_n$ with $\|t-t_n^*\|_S=\varepsilon$, 
\begin{align*}
	\hat G_n(t) & \geq G_n(t)-\eta \\
	& \geq G_n(t_n^*)+\frac{\varepsilon^2}{2}-\eta \\
	& \geq \hat G_n(t_n^*)-\eta +\frac{\varepsilon^2}{2}-\eta.
\end{align*}
Hence, chosing $\eta=\varepsilon^2/8$, we obtain that for all large enough integers $n$, with probability at least $1-2\varepsilon$, $\hat G_n(t)>\hat G_n(t_n^*)$ simultaneously for all $t\in T_n$ with $\|t-t_n^*\|_S=\varepsilon$. Corollary~\ref{cor:Rockafellar_as} yields that for all large enough integers $n$, with probability at least $1-2\varepsilon$, $\|\hat t_n-t_n^*\|_S\leq \varepsilon$. That is, $\hat t_n-t_n^*$ converges in probability to $0$.

\paragraph{\underline{Conclusion}.}
\,
We have proved that $t_n^*$ converges in distribution to $\diff^+\pi_{\Theta-\theta^*}^S(-S^{-1}\nabla\Phi(\theta^*);Z)$ for some Gaussian random variable $Z$ and that $\hat t_n-t_n^*$ converges to zero in probability, as $n\to\infty$. Hence, Slutsky's theorem implies the desired result. 

\end{proof}

\begin{remark}[On the joint asymptotic distribution of $M$-estimators]
By augmenting the loss function of $M$-estimators falling in the framework of the previous theorem, we can also easily obtain the joint asymptotic distribution of such $M$-estimators. Indeed, consider two such $M$-estimators associated with loss functions (that are convex in their second argument) $\phi_1:E\times\Theta_0\to\R$ and $\phi_2:E\times\Xi_0\to\R$ where $\Theta_0$ and $\Xi_0$ are open convex subsets of $\R^{d_1}$ and $\R^{d_2}$ respectively, for some positive integers $d_1$ and $d_2$. Consider two constraint sets $\Theta\subseteq\Theta_0$ and $\Xi\subseteq\Xi_0$ that are closed and convex. Denote these two $M$-estimators by $\hat\theta_n$ and $\hat\xi_n$, and their population counterpart by $\theta^*$ and $\xi*$, respectively.

Assume that all assumptions of Theorem~\ref{thm:asympt_norm} are satisfied. Let $\phi:E\times \Theta_0\times\Xi_0\to\R$ be the loss function defined by $\phi(x,\theta,\xi)=\phi_1(x,\theta)+\phi_2(x,\xi)$, for all $x\in E$, $\theta\in\Theta_0$ and $\xi\in\Xi_0$. Then, the pair $(\hat\theta_n,\hat\xi_n)$ is the $M$-estimator obtained with loss function $\phi$ and constraint set $\Theta\times\Xi$ (which is a closed, convex subset of $\Theta_0\times\Xi_0$), with population counterpart $(\theta^*,\xi^*)$, and it is easy to check that all assumptions of Theorem~\ref{thm:asympt_norm} are then met for this $M$-estimator. Hence, we obtain that
$$\sqrt n\left(\begin{matrix}
\hat\theta_n-\theta^*\\\hat\xi_n-\xi^*
\end{matrix}\right) \xrightarrow[n\to\infty]{}\diff^+\pi_{(\Theta-\theta^*)\times (\Xi-\xi^*)}^S\left(-S^{-1}\left(\begin{matrix}
\nabla\Phi_1(\theta^*) \\ \nabla\Phi_2(\xi^*)
\end{matrix}\right);Z\right)$$
where 
\begin{itemize}
	\item $\Phi_1(\theta)=\E[\phi_1(X_1,\theta)]$ and $\Phi_2(\xi)=\E[\phi_2(X_1,\xi)]$ for all $\theta\in\Theta_0$ and $\xi\in\Xi_0$,
	\item $\DS S=\left(\begin{matrix}
\nabla^2\Phi_1(\theta^*) & 0 \\ 0 & \nabla^2\Phi_2(\xi^*)
\end{matrix}\right)$,
	\item $Z\sim\mathcal N_{d_1+d_2}(0,S^{-1}BS^{-1})$ and, finally, 
	\item $B$ is the covariance matrix of the vector $(d_1+d_2)$-dimensional vector $\DS \left(\begin{matrix}
g_1(X_1,\theta^*) \\ g_2(X_1,\xi^*)
\end{matrix} \right)$ where $g_1(X_1,\theta^*)$ is a subgradient of $\phi_1(X_1,\cdot)$ at $\theta^*$ and $g_2(X_1,\xi^*)$ is a subgradient of $\phi_2(X_1,\cdot)$ at $\xi^*$.
\end{itemize}

In particular, in the absence of constraints (that is, when $\pi_{(\Theta-\theta^*)\times (\Xi-\xi^*)}^S$ is the identity), these estimators are jointly asymptotically normal.
\end{remark}

\vspace{4mm}

In the proof of Theorem~\ref{thm:asympt_norm}, the convergence that we obtained in \eqref{eqn:pointwisePconv} actually holds in the $L^2$ sense (see \eqref{eqn:convL2}). Therefore, Corollary~\ref{cor:Rockafellar_Lp} implies uniform convergence on all compact subsets in the $L^2$ sense. Yet, it is not clear, from there, how to proceed and prove that $\hat t_n-t_n^*\xrightarrow[n\to\infty]{} 0$ in $L^2$. Proving this convergence would yield an exact asymptotic quantification of the mean squared error of $\hat\theta_n$, since, it would yield that 
$$n\E[\|\hat\theta_n-\theta^*\|^2]\xrightarrow[n\to\infty]{}\E[\|\diff^+\pi_{\Theta-\theta^*}^S(-S^{-1}\nabla\Phi(\theta^*);Z)\|^2]$$
where $Z$ is a Gaussian vector as in the theorem. We leave the following question open:
\begin{openquestion}
	Is it true that under the assumptions of Theorem~\ref{thm:asympt_norm}, for all large enough $n$, $\hat\theta_n$ has two moments, and that 
	$$n\E[\|\hat\theta_n-\theta^*\|^2]\xrightarrow[n\to\infty]{}\E[\|\diff^+\pi_{\Theta-\theta^*}^S(-S^{-1}\nabla\Phi(\theta^*);Z)\|^2]?$$
\end{openquestion}

We can further obtain a more precise asymptotic description of $\hat\theta_n$ in the following setup. Let $T$ be the tangent cone to $\Theta$ at $\theta^*$. Assume that $T$ contains a non-trivial linear space $L$. For instance, this happens when $\Theta$ is a convex polytope and $\theta^*$ is not a vertex of $\Theta$. Then, $T\cap L^\perp$ is a convex cone and one can easily check that $T$ can be decomposed as $T=L+(T\cap L^\perp)$. Moreover, any $t\in T$ can be uniquely decomposed as $t=u+v$ where $u\in L$ and $s\in T\cap L^\perp$ ($u$ is given by $\pi_{L}(t)$).  
Recall that $\nabla\Phi(\theta^*)^\top t\geq 0$ for all $t\in T$, by Lemma~\ref{lem:FOC_gen}. In particular, $\nabla\Phi(\theta^*)^\top u=0$ for all $u\in L$.  
Recall the definition of the function $F_n$ from \eqref{eq:Fn}. As a consequence of the convergence proved in \eqref{eqn:pointwisePconv}, we obtain the following result.

\begin{lemma} \label{lem:pointwisePconv_uv} 
	Let $R>0$ be fixed. Then, 
\begin{align*}\sup_{\substack{u\in L:\|u\|\leq R \\v\in T\cap L^\perp:\|v\|\leq R}} &  \bigg| n\left(\Phi_n\left(\theta^*+\frac{u}{\sqrt n}+\frac{v}{n}\right)-\Phi_n(\theta^*)\right) \\
& \quad \quad \quad -u^\top Z_n-\frac{1}{2}u^\top\nabla^2\Phi(\theta^*)u-v^\top\nabla\Phi(\theta^*)\bigg| \xrightarrow[n\to\infty]{}0
\end{align*}
in probability, where $\DS Z_n=\sqrt n \left(\frac{1}{n}\sum_{i=1}^n g(X_i,\theta^*)-\nabla\Phi(\theta^*)\right)$.
\end{lemma}

\begin{proof}
	Recall that $\DS\sup_{t\in B(0,2R)}|F_n(t)|\xrightarrow[n\to\infty]{}0$ in probability, by \eqref{eqn:pointwisePconv}. Now, for all large enough $n$, $\|u+v/\sqrt n\|\leq 2R$ for all $u\in L, v\in L^\perp$ with $\|u\|\leq R$ and $\|v\|\leq R$. Hence, 
$\DS\sup_{\substack{u\in L:\|u\|\leq R \\ v\in L^\perp:\|v\|\leq R}}|F_n(u+v/\sqrt n)|\xrightarrow[n\to\infty]{}0$ in probability. Now, for all $u\in L$ and $v\in L^\perp$, 
\begin{align*}
F_n(u+v/\sqrt n) & = n\left(\Phi_n(\theta^*+u/\sqrt n+v/n)-\Phi_n(\theta^*)\right)-\frac{u^\top}{\sqrt n}\sum_{i=1}^n g(X_i,\theta^*)-\frac{1}{2}u^\top\nabla^2\Phi(\theta^*)u \\
& \quad \quad \quad - \frac{v^\top}{n}\sum_{i=1}^n g(X_i,\theta^*) -\frac{u^\top}{\sqrt n}\nabla^2\Phi(\theta^*)v-\frac{1}{2n}v^\top\nabla^2\Phi(\theta^*)v.
\end{align*}
The first term on the second line of this display converges in probability to $v^\top\nabla\Phi(\theta^*)$ uniformly in $v$ with $\|v\|\leq R$ by the law of large numbers, and the last two terms go (deterministically) to $0$ uniformly in $u,v$ with $\|u\|,\|v\|\leq R$. The desired result then follows from the fact that $u^\top\nabla\Phi(\theta^*)=0$ for all $u\in L$. 
\end{proof}

Now, based on Lemma~\ref{lem:pointwisePconv_uv} and following the same reasoning as in the proof of Theorem~\ref{thm:asympt_norm}, we obtain the following theorem.

\begin{theorem} \label{thm:asymp_uv}
	Recall the notation of Theorem~\ref{thm:asympt_norm}. Furthermore, let $T$ be the tangent cone to $\Theta$ at $\theta^*$ and $C=\bar T$ be the supporting cone. Let $L$ be the largest linear subspace that is contained in $T$. 
	Assume that $\Phi$ is twice differentiable at $\theta^*$, that $g(\cdot,\theta^*)\in L^2(P)$. Further assume that for all $v\in (C\cap L^\perp)\setminus\{0\}$, $v^\top\nabla\Phi(\theta^*)>0$. 
Then, 
$$\sqrt n\pi_L(\hat\theta_n-\theta^*) \xrightarrow[n\to\infty]{} \pi_L^S(Z)$$
in distribution, where $Z\sim\mathcal N_d(0,S^{-1}BS^{-1})$ and
$$n\pi_{L^\perp}(\hat\theta_n-\theta^*)\xrightarrow[n\to\infty]{}0$$
in probability.
\end{theorem}

Under the assumptions of this theorem, convergence of $\sqrt\pi_L(\hat\theta_n-\theta^*)$ could already be deduced from Theorem~\ref{thm:asympt_norm}. (recall that $B$ is the covariance matrix of $g(X_1,\theta^*)$). However, convergence of $n\pi_{L^\perp}(\hat\theta_n-\theta^*)$ is stronger than what is given in Theorem~\ref{thm:asympt_norm}, which only yields convergence to $0$ at rate $n^{-1/2}$.
When $\Phi$ is differentiable at $\theta^*$, the assumptions of Theorem~\ref{thm:First_order_asymp} are a particular case of Theorem~\ref{thm:asymp_uv} ($L$ must be $\{0\}$ there), and the conclusion of Theorem~\ref{thm:First_order_asymp} was stronger in that case. However, in a sense, when $\Phi$ is differentiable at $\theta^*$, Theorem~\ref{thm:asymp_uv} provides a more complete description.

\section{Extension: Convex $U$-estimation} \label{sec:U-estimators}

The previous theory can be easily extended to more general convex empirical risks, e.g., when $\Phi_n(\theta)$ is a $U$-statistic. With the same notation as in the previous sections, fix some positive integer $k$ and let $\phi:E^k\times\Theta_0\to\R$ be symmetric and measurable in its first $k$ arguments and convex in its last. Also assume that for all $\theta\in\Theta_0$, $\phi(\cdot,\theta)\in L^1(P^{\otimes k})$, that is, $\phi(X_1,\ldots,X_k,\theta)$ is integrable. Set $\Phi(\theta)=\E[\phi(X_1,\ldots,X_k,\theta)]$ and, for all $n\geq k$, 
$$\Phi_n(\theta)=\frac{1}{{n\choose k}}\sum_{1\leq i_1<\ldots<i_k\leq n}\phi(X_{i_1},\ldots,X_{i_k},\theta).$$ 

Estimators obtained by minimizing such empirical risks are called $U$-estimators. Some relevant examples include:

\begin{enumerate}
	\item Location estimators through depth functions: Let $E=\Theta_0=\Theta=\R^d$, $k=d$ and $\phi(x_1,\ldots,x_d,\theta)$ be the volume of the $d$-dimensional simplex spanned by $x_1,\ldots,x_d,\theta$, for all $x_1,\ldots,x_d,\theta\in\R^d$. The minimizers of $\Phi$ are then called Oja's population medians \cite{oja1983descriptive}. Note that $\phi(x_1,\ldots,x_d,\theta)$ is the absolute value of an affine function of $\theta$, hence, it is convex in $\theta$. We recover consistency and asymptotic normality of Oja's empirical medians (see \cite{oja1985asymptotic}) as particular cases of our asymptotic theorems (see below for $U$-estimators). More generally, we refer to \cite{zuo2000general} for other definitions of medians that are $U$-estimators associated with depth functions. 
	\item Let $E=\R$ and $\Theta\subseteq\Theta_0=\R$ and $k\geq 1$. \cite{minsker2023u} proposes a version of the median of mean estimator defined as a $U$-estimator obtained by computing an empirical median of all empirical averages of the form $\frac{1}{k}\sum_{i\in I}X_i$, for $I\subseteq\{1,\ldots,n\}$ of size $k$. That is, $\phi(x_1,\ldots,x_k,\theta)=\left|\frac{x_1+\ldots+x_k}{k}-\theta\right|$, for all $x_1,\ldots,x_k,\theta\in\R$. The difference with standard median of mean estimators \cite{nemirovskij1983problem,lerasleandoliveira2011,lecue2020robust} is that in \cite{minsker2023u}, all possible subsamples of size $k$, with overlaps, are considered. Other frameworks, such as geometric medians of means in multivariate settings \cite{minsker2015geometric} can be considered as well. Note that in \cite{minsker2023u}, the order $k$ of the $U$-process is allowed to grow with the sample size $n$ - we do not consider this setup here and leave it for future work. 
	\item More generally, aggregation of estimators that are based on overlapping subsamples, e.g., random forests \cite{breiman2001random} or bagging \cite{breiman1996bagging}, which have attracted lots of interest in modern machine learning. 
	\item Scatter estimation and robustness: Let $E=\R$, $\Theta_0=\R$, $k=2$ and $\phi(x_1,x_2,\theta)=\ell(|x_1-x_2|^p-\theta)$ where $p\geq 1$ and $\ell=\R\to\R$ is a convex function. When $p=2$ and $\ell(u)=u^2, u\in\R$, $\hat\theta_n$ is simply twice the empirical variance of $X_1,\ldots,X_n$ and if $\ell=h_c$ for some $c>0$ (recall the definition of $h_c$ from Section~\ref{sec:Preliminaries}), we obtain a robust version of the empirical variance. If now $p=1$ and $\ell(u)=u^2, u\in\R$, we obtain Gini's mean absolute difference, while if $\ell=|\cdot|$, we obtain a proxy to a median absolute deviation (and intermediate robust versions if $\ell=h_c$ for some $c>0$). In higher dimensions, one recovers the empirical covariance matrix of $X_1,\ldots,X_n$ by setting $\phi(x_1,x_2,\theta)=\tr(((x_1-x_2)(x_1-x_2)^\top-\theta)^2)$, for all $\theta\in\R^{d\times d}\approx \R^{d^2}$ and $x_1,x_2\in\R^d$. Robust versions can be defined by taking the square root of the above, or applying Huber's loss $h_c$ for some $c>0$. 
	\item Empirical risk minimization where the choice of loss function itself depends on the data (e.g., for data driven procedures), see, e.g., \cite{sherman1994u}.

\end{enumerate}

Note that $U$-statistics depending on a parameter (here, $\Phi_n(\theta), \theta\in\Theta_0$) have been studied as $U$-processes, see, e.g., \cite{nolan1987u,nolan1988functional,arcones1993limit}.  Here, we first recall the classical law of large numbers and central limit theorem for $U$-statistics.

\begin{theorem}{Law of large numbers for $U$-statistics \cite[Theorem 8.6]{henze2024asymptotic}} \label{thm:LLN-U}
Let $h:E^k\to\R^d$ be a symmetric, measurable map satisfying $h\in L^1(P^{\otimes k})$. Then, 
$$\frac{1}{{n\choose k}}\sum_{1\leq i_1<\cdot<i_k\leq n}h(X_{i_1},\ldots,X_{i_k})\xrightarrow[n\to\infty]{}\E[h(X_1,\ldots,X_k)]$$ 
almost surely.
\end{theorem}

\begin{theorem}{Central limit theorem for multivariate $U$-statistics \cite[Theorem 7.1]{hoeffding1992class}, \cite[Theorem 8.9]{henze2024asymptotic}} \label{thm:CLT_U}
Let $h:E^k\to\R^d$ be a symmetric, measurable map satisfying $h\in L^2(P^{\otimes k})$. Let $\Sigma$ be the covariance matrix of $\E[h(X_1,\ldots,X_k)|X_1]$.\footnote{$\Sigma$ can also be written as $\E[h(X_1,X_2,\ldots,X_k)h(X_1,X_2',\ldots,X_k')^\top]-\E[h(X_1,\ldots,X_k)]\E[h(X_1,\ldots,X_k)]^\top$, that is, the covariance of the random vectors $h(X_1,X_2,\ldots,X_k)$ and $h(X_1,X_2',\ldots,X_k')$, where $X_2',\ldots,X_k'$ are such that $X_1,X_2,\ldots,X_k,X_2',\ldots,X_k'$ are iid.} For all $n\geq k$, let $\DS U_n=\frac{1}{{n\choose k}}\sum_{1\leq i_1<\cdot<i_k\leq n}h(X_{i_1},\ldots,X_{i_k})$. Then,
$$\sqrt n(U_n-\E[h(X_1,\ldots,X_k)])\xrightarrow[n\to\infty]{} \mathcal N_d(0,k^2\Sigma)$$
in distribution.
\end{theorem}

Theorem~\ref{thm:consist} obviously remains true in the context of $U$-estimation with convex loss. Proposition~\ref{prop:non-smooth}, Theorems~\ref{thm:non-diff} and \ref{thm:First_order_asymp} require more care but also remain true in this context. Proofs are deferred to Section~\ref{sec:appendix_extra_proofs}. Below, we rewrite Theorem~\ref{thm:asympt_norm} for $U$-estimators, where an extra multiplicative factor $k$ appears in the limit, accounting for the dependence of the terms in the new definition of $\Phi_n$. 

\begin{theorem}{Asymptotic distribution for $U$-estimators} \label{thm:asymp_norm_U}
Let $g:E^k\times\Theta_0\to\R^d$ be a measurable selection of subgradients of $\phi$. Assume the following:
\begin{itemize}
	\item[(i)] $\Phi$ has a unique minimizer $\theta^*$ in $\Theta$, it is twice differentiable at $\theta^*$ and $S:=\nabla^2\Phi(\theta^*)$ is positive definite;
	\item[(ii)] $g(\cdot,\theta^*)\in L^2(P^{\otimes k})$;
	\item[(iii)] $\pi_{\Theta-\theta^*}^S$ has directional derivatives at $-S^{-1}\nabla\Phi(\theta^*)$. 
\end{itemize}
Then, 
$$\sqrt n(\hat\theta_n-\theta^*)\xrightarrow[n\to\infty]{} k\diff^+\pi_{\Theta-\theta^*}^S(S^{-1}\nabla\Phi(\theta^*);Z)$$
in distribution, where $Z\sim\mathcal N_d(0,S^{-1}BS^{-1})$ and $B=\var(\E[g(X_1,\ldots,X_k,\theta^*)|X_1])$. 
\end{theorem}
Note the extra $k$ factor in the limit in distribution.

\section{Conclusion and future directions} \label{sec:conclusion}

We have established the asymptotic properties of constrained $M$-estimators with a convex loss and a convex set of constraints, under minimal assumptions. In this work, asymptotics are only relative to the sample size $n$, while the dimension $d$ is kept fixed. 

In large dimensional problems, asymptotic theory can be approached from different angles. First, one may look at asymptotic distributions of low-dimensional projections of the $M$-estimator. For instance, in the context of linear regression, \cite{bellec2022asymptotic} proves the asymptotic normality of single coordinates of penalized $M$-estimators when the ratio $d/n$ goes to some fixed, positive constant. 
A second angle consists of looking at the full, joint distribution of (a rescaled version of) the $M$-estimator $\hat\theta_n$, and prove that, for some distribution $Q_d$ in $\R^d$, some specified distance (e.g., an integral probability metric) between the distribution of $\hat\theta_n$ and $Q_d$ goes to $0$ as $n,d\to\infty$ in a certain manner. When $\hat\theta_n$ is simply the sample mean of $X_1,\ldots,X_n$, such an approach has been studied and called \textit{high dimensional central limit theorems} \cite{chernozhukov2017central,fang2021high}. However, to the best of our knowledge, such results do not exist for other $M$-estimators, even with convex loss. 

In the context of $U$-estimators, we have also let the order $k$ of the $U$-process be fixed. However, it may be relevant to also let $k$ grow with the sample size (e.g., for median-of-means procedures). While the asymptotics of $U$-statistics with increasing order have been studied only recently \cite{diciccio2022clt}, we leave this direction for future work on $U$-estimation.

\bibliographystyle{plain}
\bibliography{Biblio}

\appendix

\section{On convex functions and their subgradients} \label{App:convex}

In this section, we gather useful properties on subgradients of convex functions. Most of these properties are classical and we include their proofs for completeness. 

\begin{lemma} \label{lemma:subgradients_interior}
	Let $G_0\subseteq \R^d$ be a convex set with non-empty interior and $f:G_0\to\R$ be a convex function. Let $t_0\in\inter(G_0)$ and $\varepsilon>0$ be such that $B(t_0,\varepsilon)\subseteq \inter(G_0)$. Then, for all vectors $u\in\R^d$, $u\in\partial f(t_0)$ if and only if $f(t)\geq f(t_0)+u^\top (t-t_0)$, for all $t\in B(t_0,\varepsilon)$. 
\end{lemma}

\begin{proof}
	The left-right implication trivially follows the definition of subgradients. Assume now that a vector $u\in\R^d$ satisfies the right property. Fix $t\in G_0$ be arbitrary and let us show that $f(t)\geq f(t_0)+u^\top(t-t_0)$. This is clear by assumption if $t\in B(t_0,\varepsilon)$, so let us assume that $t\notin B(t_0,\varepsilon)$. Let $\lambda=\varepsilon/\|t-t_0\|\in (0,1)$ and $t_\lambda:=t_0+\lambda (t-t_0)\in B(t_0,\varepsilon)$. Then, by assumption, $f(t_\lambda)\geq f(t_0)+u^\top(t_\lambda-t_0)=f(t_0)+\lambda u^\top(t-t_0)$. Moreover, by convexity of $h$, $f(t_\lambda)\leq (1-\lambda)f(t_0)+\lambda f(t)$. Rearranging yields the desired inequality. 
\end{proof}

\begin{lemma} \label{lem:prop_subgrad}
	Let $G_0\subseteq\R^d$ be a convex set and $f:G_0\to\R$ be any convex function. Then, for all $t_0\in G_0$, $\partial f(t_0)$ is closed. Moreover, if $t_0\in\inter(G_0)$, then, $\partial f(t_0)$ is non-empty and compact.
\end{lemma}

\begin{proof}
\,
\paragraph{\underline{$\partial f(t_0)$ is closed}.} 

	Fix $t_0\in G_0$ and let $(u_n)_{n\geq 1}$ be a sequence of subgradients of $h$ at $t_0$, assumed to converge to some $u\in\R^d$. For all $t\in G_0$ and for all $n\geq 1$, $f(t)\geq f(t_0)+u_n^\top (t-t_0)$. Taking the limit as $n\to\infty$ shows that $u\in\partial f(t_0)$. Hence, $\partial f(t_0)$ is closed.
\vspace{2mm}

	\paragraph{\underline{$\partial f(t_0)$ is nonempty if $t_0\in\inter(G_0)$}.}  
	Let $F=\{(t,y)\in G_0\times \R: f(t)\leq y\}$ be the epigraph of $h$, which is a convex subset of $\R^{d+1}$. Let $t_0\in\inter(G_0)$. The point $(t_0,f(t_0))$ is a boundary point of $F$ since $(t_0,f(t_0)-\varepsilon)\notin F$ for all $\varepsilon>0$. Let $H\subseteq\R^{d+1}$ be a supporting hyperplane of $F$ at this point and let $v=(v_1,v_2)$ be a (non-zero) outward pointing normal vector to $H$, where $v_1\in\R^d$ and $v_2\in\R$. This simply means that for all $(t,y)\in F$, the scalar product between $v$ and $(t,y)-(t_0,f(t_0))$ is non-positive. That is, 
\begin{equation} \label{proof:3.1}
v_1^\top (t-t_0)+v_2(y-f(t_0))\leq 0
\end{equation} 
for all $t\in G_0$ and $y\geq f(t)$.
	
	Let us show that necessarily, $v_2<0$. First, assume, for the sake of contradiction, that $v_2= 0$. Then, $v_1\neq 0$, because we have assumed that $v=(v_1,v_2)$ is non-zero. Since $t_0\in\inter(G_0)$, there exists $t\in G_0$ such that $v_2^\top(t-t_0)>0$ (take $t=t_0+\varepsilon v_2$ for any small enough $\varepsilon>0$), which contradicts \eqref{proof:3.1}. Hence, $v_2\neq 0$. Now, fixing any $t\in G_0$ and taking $y\to\infty$ in \eqref{proof:3.1} shows that $v_2$ must be negative. Now, taking $y=f(t)$ in  \eqref{proof:3.1} yields, for all $t\in G_0$,
\begin{equation*}
	f(t)\geq f(t_0)-v_2^{-1}v_1^\top (t-t_0).
\end{equation*}
That is, $-v_2^{-1}v_1$ is a subgradient of $h$ at $t_0$, so $\partial f(t_0)\neq\emptyset$. 
\vspace{2mm}
	\paragraph{\underline{$\partial f(t_0)$ is compact}.} It is now enough to check that for $t_0\in \inter(G_0)$, $\partial f(t_0)$ is bounded. Fix $\varepsilon>0$ such that $B(t_0,\varepsilon)\subseteq \inter(G_0)$. Since $h$ is continuous on $\inter(G_0)$, it is bounded on the compact set $B(t_0,\varepsilon)$. Let $M:=\max_{t\in B(t_0,\varepsilon)}f(t)$. Let $u\in \partial f(t_0)$ and assume that $u\neq 0$. Then, letting $t=t_0+\varepsilon u/\|u\|\in B(t_0,\varepsilon)$, the definition of subgradients yields that $M-f(t_0)\geq f(t)-f(t_0)\geq u^\top (t-t_0)=\varepsilon\|u\|$. Hence, $\partial f(t_0)\subseteq B(0,(M-f(t_0))/\varepsilon)$.
	
\end{proof}

If $t$ is a boundary point of $G_0$, then $\partial f(t)$ might be empty. This is the case, for instance, for $G_0=\R^+$ and $f:t\in\R^+\mapsto -\sqrt t$, which does not have any subgradient at $0$.

\begin{lemma} \label{lem:diff_subgrad}
	Let $G_0\subseteq\R^d$ be a convex set and $f:G_0\to\R$ be a convex function. Let $t_0\in\inter(G_0)$ and assume that $h$ is differentiable at $t_0$. Then, $\partial f(t_0)=\{\nabla f(t_0)\}$.
\end{lemma}

That is, if $f$ is differentiable at some interior point of its domain, then its gradient is the only subgradient at that point. This property does not hold if $t_0$ is a boundary point. For instance, let $f:t\in\R^+\mapsto 0$, which is convex. Then, while it is differentiable at $0$, $\partial f(0)=\R^-$. 

\begin{proof}
	Let $u\in \partial f(t_0)$, where $t_0\in\inter(G_0)$. Then, for all $v\in\R^d$ and all small enough $\varepsilon>0$,
$$f(t_0+\varepsilon v)-f(t_0)\geq \varepsilon u^\top v.$$
Dividing by $\varepsilon$ and taking the limit as $\varepsilon\to 0$ yields
$$\nabla f(t_0)^\top v\geq u^\top v.$$
Since this must hold for any $v\in\R^d$, one readily obtains that $u=\nabla f(t_0)$. 
\end{proof}

\begin{lemma} \label{lem:seq_subgrad}
Let $G_0\subseteq \R^d$ be a convex set and $f:G_0\to\R$ be a convex function. Let $t_0\in G_0$ and assume that $h$ is differentiable at $t_0$. Let $(t_n)_{n\geq 1}$ be any sequence of points in $G_0$ converging to $t_0$. For all $n\geq 1$, let $u_n\in \partial f(t_n)$. Then, $u_n\xrightarrow[n\to\infty]{}\nabla f(t_0)$.
\end{lemma}

\begin{proof}
	Let $\varepsilon>0$ be such that $B(t_0,\varepsilon)\subseteq\inter(G_0)$. A similar argument as in the proof of compactness of $\partial f(t_0)$ in Lemma~\ref{lem:prop_subgrad}  yields that $\bigcup_{t\in B(t_0,\varepsilon)}\partial f(t)$ is bounded, so the sequence $(u_n)_{n\geq 1}$ must be bounded. Therefore, it is sufficient to prove that any converging subsequence must converge to $\nabla f(t_0)$. Since we could simply relabel the indices of the sequence, let us simply assume that $u_n\xrightarrow[n\to\infty]{} u$ for some $u\in\R^d$. For all $t\in G_0$ and all $n\geq 1$, 
$$f(t)\geq f(t_n)+u_n^\top (t-t_n).$$
Recall that $h$ is continous on $\inter(G_0)$, so taking the limit as $n\to\infty$ in the previous display gives
$$f(t)\geq f(t_0)+u^\top (t-t_0),$$
so $u\in\partial f(t_0)$. Lemma~\ref{lem:diff_subgrad} implies that $u=\nabla f(t_0)$.
\end{proof}

The following lemma is more general and will allow to connect subgradients and directional derivatives.

\begin{lemma} \label{lem:subgrad_subseq}
	Let $G_0\subseteq \R^d$ be a convex set and $f:G_0\to\R$ be a convex function. Let $x\in\inter(G_0)$ and let $(x_n)_{n\geq 1}$ be a sequence of points in $\inter(G_0)$ converging to $x$. For each $n\geq 1$, let $u_n\in\partial f(x_n)$. Then, the sequence $(u_n)_{n\geq 1}$ is bounded and any of its converging subsequences converges to some element of $\partial f(x)$. 
\end{lemma}

\begin{proof}
	Let $\varepsilon>0$ satisfying $B(x,\varepsilon)\subseteq \inter(G_0)$. Without loss of generality, let us assume that $x_n\in B(x,\varepsilon)$ for all $n\geq 1$. Convexity of $f$ yields that it is locally Lipschitz \cite[Theorem 1.5.3]{schneider2014convex}, and hence, there is some $L>0$ such that $\|u\|\leq L$ for all $u\in \partial f(y), y\in B(x,\varepsilon)$. Hence, $(u_n)_{n\geq 1}$ is bounded. 
	
Now, consider a converging subsequence of $(u_n)_{n\geq 1}$ which, up to renumbering, we still denote by $(u_n)_{n\geq 1}$. Let $u$ be its limit. Then, for all $y\in G_0$ and $n\geq 1$, 
$$f(y)\geq f(x_n)+u_n^\top (y-x_n).$$
Since $f$ is continuous at $x$, all terms have a limit as $n\to\infty$ and we obtain, for all $y\in G_0$, 
$$f(y)\geq f(x)+u^\top (y-x).$$
That is, $u\in\partial f(x)$.
\end{proof}

\begin{lemma} \label{lem:dirder_support}
	Let $G_0\subseteq \R^d$ be a convex set and $f:G_0\to\R$ be a convex function. Let $x\in\inter(G_0)$. Then, for all $t\in\R^d$,
	$$\diff^+f(x;t)=h_{\partial f(x)}(t)$$
where we recall that $h_{\partial f(x)}$ is the support function of $\partial f(x)$. 
\end{lemma}

\begin{proof}
Let us first check that for all $u\in\partial f(x)$, $u^\top t\leq \diff^+f(x;t)$. To obtain this, note that by definition of $u$, we have that $f(x+\varepsilon t)\geq f(x)+\varepsilon t^\top u$ for all $\varepsilon>0$ and, hence, by rearranging and taking the limit as $\varepsilon\to 0$, $\diff^+f(x;t)\geq u\top t$. Now, let us simply check the existence of $u\in\partial f(x)$ satisfying $u^\top t=\diff^+f(x;t)$: This will end the proof. 

For all large enough integers $n$ (so $x+t/n\in \inter(G)$), let $u_n\in \partial f(x+t/n)$. Then, $f(x)\geq f(x+t/n)-u_n^\top t/n$ which, after rearranging, gives:
$$n(f(x+t/n)-f(x))\leq u_n^\top t.$$
The left hand side goes to $\diff^+f(x,t)$ and, by Lemma~\ref{lem:subgrad_subseq}, the right hand side has a subsequence that goes to $u^\top t$ for some $u\in\partial f(x)$. We thus obtain that $\diff^+f(x;t)\leq u^\top t$, which is what we aimed for. 
\end{proof}

\begin{lemma}[First order condition] \label{lem:FOC_gen}
	Let $G_0$ be an open convex set and $G\subseteq G_0$ be closed and convex. Let $f:G_0\to\R$ be a convex function and $x_*\in G$. Let $C$ be the support cone to $G$ at $x^*$. Then, 
$$f(x)\geq f(x^*), \forall x\in G \quad \iff \quad h_{\partial f(x^*)}(t)\geq 0,\forall t\in C.$$ 
\end{lemma}

In particular, we recover the standard first order condition if $f$ is differentiable at $x^*$, that is, $x^*$ is a minimizer of $f$ on $G$ if and only if $\nabla f(x^*)\top t\geq 0$ for all $t\in C$. 

\begin{proof}
	Let $T$ be the tangent cone to $G$ at $x^*$, so $C$ is the closure of $T$.
It is clear that $x^*$ is a minimizer of $f$ on $G$ if and only if $\diff^+f(x^*;t)\geq 0$, that is, $h_{\partial f(x^*)}(t)\geq 0$ by Lemma~\ref{lem:dirder_support}.
The result follows from the continuity of $h_{\partial f(x^*)}$. 
\end{proof}

\begin{lemma}[Monotonicity of subgradients] \label{lem:monoton_subgrad}
	Let $G_0\subseteq \R^d$ be a convex set and $f:G_0\to\R$ be a convex function. Then, for all $t_1,t_2\in G_0$ and $u_1\in\partial f(t_1), u_2\in\partial f(t_2)$, we have $(t_1-t_2)^\top (u_1-u_2)\geq 0$. 
\end{lemma}

For differentiable, convex functions on $\R$, this lemma simply says that the derivative is non-decreasing. 

\begin{proof}
	By definition of subgradients,
$$f(t_1)\geq f(t_0)+u_0^\top (t_1-t_0)$$
and
$$f(t_0)\geq f(t_1)+u_1^\top (t_0-t_1).$$
Adding these two inequalities yields the result. 
\end{proof}

\begin{lemma} \label{lem:diff_subgrad2}

Let $f$ be a random convex function defined on a convex set $G_0\subseteq\R^d$ and let $x_0\in \inter(G_0)$.Assume that for all $x\in G_0$, $f(x)$ is integrable and let $F(x)=\E[f(x)]$. Then, for all $t\in\R^d$, 
$$\E[\diff^+f(x_0;t)]=\diff^+F(x_0;t).$$
In particular, if $F$ is differentiable at $x_0$, so is $f$ almost surely. 

\end{lemma}

\begin{proof}

Fix $t\in\R^d$. For simplicity (and without loss of generality), assume that $x_0-t,x_0+t\in G_0$. First, we have that $\diff^+f(x_0;t)=\lim_{n\to\infty} n(f(x_0+t/n)-f(x_0))$ almost surely. Moreover, convexity of $f$ yields that:
\begin{itemize}
	\item $n(f(x_0+t/n)-f(x_0))$ is non-increasing with $n$;
	\item $f(x_0)-f(x_0-t)\leq n(f(x_0+t/n)-f(x_0)) \leq f(x_0+t)-f(x_0)$
\end{itemize} 
where both bounds in the last display are integrable. Therefore, monotone convergence implies that 
$$\E[\diff^+f(x_0;t)] = \lim_{n\to\infty} \E[n(f(x_0+t/n)-f(x_0))] = \lim_{n\to\infty} n(F(x_0+t/n)-F(x_0)) =  \diff^+F(x_0;t).$$

Now, let us assume that $F$ is differentiable at $x_0$. 
Since $f$ is convex, in order to show that it is almost surely differentiable at $x_0$, it is enough to show that it has partial derivatives at $x_0$ along all canonical basis directions with probability $1$ (see \cite[Theorem 1.5.8]{schneider2014convex}). That is, we need to show that with probability $1$, for all canonical basis vectors $e$, it holds that $\diff^+f(x_0;e)=\diff^+f(x_0,-e)$. Convexity of $f$ yields that the right hand side is larger or equal to the left hand side, and the first part of this lemma implies that the expected difference is zero. Hence, both sides are equal with probability $1$, which concludes the proof. 

\end{proof}

\begin{remark}
	In the previous lemma, convexity of the random function $f$ is key. For instance, let $f_0(\theta)=|\theta|$ and $f_1(\theta)=-|\theta|$, for all $\theta\in\R$ ($d=1$ here). Set $f=f_I$, where $I$ is a Bernoulli random variable with $P(I=0)=P(I=1)=1/2$. Then, with probability $1$, $f_I$ is not differentiable at $0$. Yet, $F$ is the constant function equal to $0$, which is differentiable at $0$. 
\end{remark}

\section{On metric projections} \label{App:projections}

The following lemma is a very standard result on projections on closed, convex sets in Euclidean spaces. 
We choose to state it here with our notation for the ease of the reader. 

\begin{lemma} \label{lem:charact_proj}

Let $G\subseteq\R^d$ be a non-empty, closed, convex set and $S\in\R^{d\times d}$ be symmetric, positive definite. Then, for all $z\in\R^d$, $\pi_G^S(z)$ is the unique $z^*\in G$ satisfying 
$$\langle x-z^*,z-z^*\rangle_S\leq 0, \quad \forall x\in G.$$
In particular, $z-\pi_G^S(z)$ is in the normal cone to $G$ at $\pi_G^S(z)$ with respect to $S$. 
Moreover, $\pi_G^S$ is non-expansive with respect to $\|\cdot\|_S$, that is, for all $z,z'\in \R^d$, 
$$\|\pi_G^S(z)-\pi_G^S(z')\|_S\leq \|z-z'\|_S.$$

\end{lemma}

\begin{proof}
	Let $z\in \R^d$. Then, by definition of $\pi_G^S(z)$, we have, for all $x\in G$, that $\|z-\pi_G^S(z)\|_S^2\leq \|z-x\|_S^2$. Expanding these Euclidean norms and rearranging yield that $\pi_G^S(z)$ does satisfy the first inequality of the lemma. Now, assume that $z^*\in G$ also satisfies this inequality. Reverse engineering simply implies that $\|z-z^*\|_S^2\leq \|z-x\|_S^2$ for all $x\in G$, and hence, $z^*=\pi_G^S(z)$. 
	
Non-expansiveness of $\pi_G^S$ is a direct consequence of the first inequality of the lemma. Indeed, it implies both that 
$$\langle \pi_G^S(z')-\pi_G^S(z),z-\pi_G^S(z)\rangle_S\leq 0$$
and 
$$\langle \pi_G^S(z)-\pi_G^S(z'),z'-\pi_G^S(z')\rangle_S\leq 0.$$
Summing these two inequalities yields that 
\begin{align}
	\|\pi_G^S(z')-\pi_G^S(z)\|_S^2 & \leq \langle \pi_G^S(z')-\pi_G^S(z),z'-z\rangle_S \label{eqn:charact_proj} \\
	& \leq \|\pi_G^S(z')-\pi_G^S(z)\|_S \|z-z'\|_S \nonumber
\end{align}
by Cauchy--Schwarz inequality, which yields the result. 
\end{proof}

\section{On the directional differentiability of metric projections} \label{sec:appendix_diff_proj}

Here, we gather several facts on the existence of directional derivatives of metric projections and their formulas. For simplicity, we choose to state and prove all the results of this section for the standard, canonical Euclidean structure of $\R^d$, that is, for $S=I_d$. All the results and formulas extend in a straightforward manner to general symmetric, positive definite $S$. 

Let $G\subseteq\R^d$ be a non-empty, closed, convex set. Although the map $\pi_G$ is non-expansive, it does not necessarily have directional derivatives at any point. A counterexample (among others) is given in \cite{shapiro1994directionally} and can be described easily for $d=2$: Let $G$ be the convex hull of all points of the form $(\cos(\pm\pi/2^k),\sin(\pm\pi/2^k))$, for $k=0,1,2,\ldots$ and of $(1,0)$. Then, letting $x=(a,0)$ for any $a>0$, $\pi_G(x)=(1,0)$ and $\pi_G$ does not have a directional derivative at $(0,\pm 1)$. Roughly speaking, this stems from the fact that the boundary of $G$ is not twice directionally differentiable at $(1,0)$ in neither directions $(0,1)$ or $(0,-1)$. 

First, it is obvious that if $x\in\inter(G)$, then $\pi_G$ coincides with the identity map on a neighborhood of $x$, and hence, it is differentiable (and hence, it has directional derivatives) at $x$, and its Jacobian at $x$ is the identity matrix. If $x\in\partial G$ then $\pi_G$ is not always differentiable at $x$ but it has directional derivatives in all directions:

\begin{lemma}{\cite{shapiro1987differentiability},\cite[Lemma 4.6]{ZARANTONELLO1971237}} \label{lem:Zarantonello}
Let $G\subseteq\R^d$ be non-empty, closed and convex and let $x\in\partial G$. Then, $\pi_G$ has directional derivatives at $x$, given by
$$\diff^+\pi_G(x;\cdot)=\pi_{C}$$
where $C$ is the support cone to $G$ at $x$.
\end{lemma} 

In particular, $\diff^+\pi_G$ is differentiable at $x$ if and only if $\pi_C$ is a linear map, that is, $C$ is a linear subspace, if and only if $G$ is included in a strict affine subspace of $A\subseteq \R^d$ and $x$ is in the relative interior of $G$ (in that case, $A=x+C$).


When $x\in\R^d\setminus G$, we have the following sufficient condition for differentiability of $\pi_G$ at $x$.

\begin{lemma}{\cite[Lemma 1 and Theorem 2]{holmes1973smoothness}} \label{lem:diffprojCk}

Let $G\subseteq\R^d$ be non-empty, closed and convex. Let $x\in\R^d\setminus G$ and assume that the boundary of $G$ is of class $C_k$ in a neighborhood of $\pi_G(x)$, for some $k\geq 2$. Then, $\pi_G$ is of class $C^{k-1}$ in a neighborhood of $x$. 

Further assume that $\inter(G)\neq\emptyset$ and $0\in\inter(G)$ and let $\rho_G$ be the gauge function of $G$, defined by $\rho_G(y)=\inf\{\lambda>0:y\in\lambda G\}$, for all $y\in\R^d$. Then, $\rho_G$ is twice differentiable at $y:=\pi_G(x)$ with $\nabla\rho_G(y)\neq 0$ and 
\begin{equation} \label{eqn:diff_proj}
\diff\pi_G(x;\cdot)=\left(I_d+\frac{\|x-y\|}{\|\nabla\rho_G(y)\|}\pi_{(x-y)^\perp}\nabla^2\rho_G(y)\right)^{-1}\pi_{(x-y)^\perp},
\end{equation}
where we have identified linear maps with their matrices in the canonical basis.
\end{lemma}

Note that the assumption that $0\in\inter(G)$ is made with no loss of generality, since $G$ could be replaced with $G-y_0$ for some $y_0\in \inter(G)$ (and $\rho_G$ would be replaced with $\rho_{G-y_0}$ in \eqref{eqn:diff_proj}). Note also that under the assumptions of the lemma, for all $z\in\R^d$, $\diff\pi_G(x;z)\in (x-y)^\perp$. 

One can easily derive a simpler formula than \eqref{eqn:diff_proj} by identifying $\R^d$ with $\R^{d-1}\times\R$, $y$ with $(0,0)$, $x$ with $(0,-t)$ for some $t>0$ and by locally identifying $G$ with the epigraph of a twice differentiable convex map $f:\R^{d-1}\to\R$ with $f(0)=0$ and $\nabla f(0)=0$. Then, for all $z=(z_1,z_2)\in \R^{d-1}\times \R$, $\diff\pi_G(x;z)=\left((I_{d-1}+t\nabla^2f(0))^{-1}z_1,0\right).$

A simple example to have in mind is that of $G=B(0,R)$. Then, for all $x\in\R^d$ with $\|x\|>R$, we obtain 
$$\diff\pi_G(x;z)=\frac{R}{\|x\|}\pi_{x^\perp}(z), \quad \forall z\in\R^d.$$
Therefore, in that case, $\diff\pi_G(x;\cdot)$ is a rescaled version of the projection onto $G$, where the scaling factor depends on both the distance from $x$ to $G$ and the curvature of $G$ at $\pi_G(x)$. 

If $G$ is defined by smooth, convex constraints, we have the following result which guarantees the existence of directional derivatives of $\pi_G$. 

\begin{lemma}{\cite[Theorem 3.2]{shapiro2016differentiability}} \label{lem:shapiroSlater}  
Let $g_1,\ldots,g_p:\R^d\to\R$ be twice differentiable, convex functions and let $G=\{x\in\R^d:g_j(x)\leq 0, j=1,\ldots,p\}$. Assume Slater's qualification constraint: There exists $y_0\in\R^d$ with $g_j(y_0)<0$ for all $j=1,\ldots,p$. Then, $\pi_G$ has directional derivatives everywhere. 
\end{lemma}

Let us now look at further cases where $\partial G$ is not necessarily differentiable in a neighborhood of $\pi_G(x)$. First, let us explore the simple case where $G$ is a closed, convex cone and $\pi_G(x)=0$.

\begin{lemma} \label{lem:diff_proj}
	Let $C\subseteq \R^d$ be a non-empty, closed convex cone, $S\in\R^{d\times d}$ be symmetric, positive definite and $x\in \R^d$ satisfying $x^\top y\leq 0$ for all $y\in C$. Then, $\pi_C$ is directionally differentiable at $x$ and for all $z\in\R^d$, the directional derivative of $\pi_C$ at $x$ in the direction $z$ is given by 
$$\diff^+\pi_C(x;z)=\lim_{\varepsilon\downarrow 0}\frac{\pi_C(x+\varepsilon z)}{\varepsilon}=\pi_{C_x}(z).$$

\end{lemma}

Recall the notation $C_x=\{y\in C:x^\top y=0\}=C\cap x^{\perp}$. Note that, with the notation of the lemma, the assumption that $x^\top y\leq 0$ for all $y\in C$ (that is, $x$ is in the polar of $C$) implies that $\pi_C(x)=0$. 

\begin{proof}

Let $\varepsilon>0$. Since, for all $y\in\R^d$, $y\in C\iff \varepsilon y\in C$, we have that
\begin{align*}
	\pi_C(x+\varepsilon z) & = \argmin_{y\in C} \|x+\varepsilon z-y\|^2 \\
	& = \varepsilon \argmin_{y\in C} \|x+\varepsilon z-\varepsilon y\|^2 \\
	& = \varepsilon \argmin_{y\in C}\left(x^\top (z-y)+\frac{\varepsilon}{2}\|z-y\|^2 \right).
\end{align*}

If $C\subseteq x^{\perp}$, then $x^\top y=0$ for all $y\in C$ so the previous display implies that $\pi_C(x+\varepsilon z)=\varepsilon\pi_C(z)$, yielding the desired result in that case, since $C_x=C$. 

Let us now assume that $C_x$ is a strict subset of $C$. That is, there are $y\in C$ with $x^\top y<0$. Our goal is still to show that $y_\varepsilon:=\argmin_{y\in C} \left(x^\top (z-y)+\frac{\varepsilon}{2}\|z-y\|^2\right)\xrightarrow[t\to 0]{}\pi_{C_x}(z)$.

First, note that for all $\varepsilon>0$, this vector $y_\varepsilon$ is well defined by strong convexity of the function that it minimizes and the fact that $C$ is a closed convex set. Now, let $t_\varepsilon=-x^\top y_\varepsilon$. Then, by definition, $y_\varepsilon\in C_{x,t_\varepsilon}$ for all $\varepsilon>0$. Moreover, it is clear that 
\begin{equation*}
	y_\varepsilon = \argmin_{y\in C_{x,\varepsilon}} \|z-y\|^2 = \pi_{C_{x,t_\varepsilon}}(z).
\end{equation*}
So, what we have to show is that $\pi_{C_{x,t_\varepsilon}}(z)\xrightarrow[\varepsilon\to 0]{} \pi_{C_x}(z)$.

First, let us check that $t_\varepsilon\xrightarrow[\varepsilon\to\ 0]{} 0$. The fact that $t_\varepsilon\geq 0$ is clear from the fact that $x^\top y\leq 0$ for all $y\in C$ (and, in particular, for $y=y_\varepsilon$). Moreover, for all $\varepsilon>0$, 
\begin{equation*}
	x^\top z+t_\varepsilon = x^\top (z-y_\varepsilon) \leq x^\top (z-y_\varepsilon)+\frac{\varepsilon}{2}\|z-y_{\varepsilon}\|^2 \leq x^\top (z-y)+\frac{\varepsilon}{2}\|z-y\|^2
\end{equation*}
for all $y\in C$, by definition of $y_{\varepsilon}$. Choosing $y=\pi_{C_x}(z)$ yields that $t_\varepsilon\leq \frac{\varepsilon}{2}\dist(z,C_x)^2$. Therefore, $t_\varepsilon\xrightarrow[\varepsilon\to 0]{} 0$. 

Finally, in order to achieve our objective, it is sufficient to show that given any sequence $(\varepsilon_n)_{n\geq 1}$ of positive number converging to $0$, $y_{\varepsilon_n}\xrightarrow[n\to\infty]{} \pi_{C_x}(z)$. Consider such a sequence. For simplicity, let us denote by $y_n:=y_{\varepsilon_n}$, $t_n:=t_{\varepsilon_n}$ and $C_n:=C_{x,t_n}$.

Let us start by showing that the sequence $(y_n)_{n\geq 1}$ is bounded. As already mentioned above, since we have assumed that $C_x$ is a strict subset of $C$, there must exist some $y\in C$ with $\alpha:=-x^\top y>0$. For all $n\geq 1$, let $\lambda_n=t_n/\alpha$, so that $\lambda_n y\in C_n$ for all $n\geq 1$. Therefore, $\lambda_n y=\pi_{C_n}(\lambda_n y)$ and, since $\pi_{C_n}$ is non-expansive (see Lemma~\ref{lem:charact_proj}), we have that 
\begin{align*}
	\|y_n\| & \leq \|y_n-\lambda_n y\|+\lambda_n\|y\| \\
	& = \|\pi_{C_n}(z)-\pi_{C_n}(\lambda_n y)\|+\lambda_n\|y\| \\
	& \leq \|z-\lambda_n y\|+\lambda_n\|y\| \\
	& \leq \|z\|+2\lambda_n\|y\| \\
	& = \|z\|+2\alpha^{-1}t_n\|y\|
\end{align*}
which is bounded since we have shown, earlier, that $t_n\xrightarrow[n\to\infty]{}0$. The first and last inequalities above are simply the triangle inequality.

Now, since the sequence $(y_n)_{n\geq 1}$ is bounded, in order to prove that it converges to $\pi_{C_x}(z)$, it is sufficient to check that any of its converging subsequences converges to that same point. Up to renumbering, let us simply assume that $y_n\xrightarrow[n\to\infty]{} y^*$ for some $y^*\in\R^d$, and show that $y^*=\pi_{C_x}(z)$. Also, without loss of generality (since we could otherwise consider a further subsequence), let us assume that $(t_n)_{n\geq 1}$ is decreasing. First, since $y_n\in C$ for all $n\geq 1$ and $C$ is closed, it must hold that $y^*\in C$. Moreover, since $-\langle x, y_n\rangle_S=t_n\xrightarrow[n\to\infty]{} 0$ it must hold that $x^\top y^*=0$. Therefore, $y^*\in C_x$. 

Hence, by Lemma~\ref{lem:charact_proj}, in order to check that $y^*=\pi_{C_x}(z)$, it is sufficient to show that for all $y\in C_x$, $(z-y^*)^\top (y-y^*)\leq 0$. 
Let $y\in C_x$ be arbitrary. Let $(w_n)_{n\geq 1}$ be a sequence converging to $y$ and such that $w_n\in C_n$ for all $n\geq 1$. Such a sequence can be constructed, for instance, by taking $w_n$ as the unique intersection of the affine hyperplane $\{w\in\R^d=x^\top w=t_n\}$ with the segment connecting $y_1$ and $y$. Then, since $y_n=\pi_{C_n}(z)$, Lemma~\ref{lem:charact_proj} yields that $(z-y_n)^\top (w_n-y_n)\leq 0$, for all $n\geq 1$. Taking the limit as $n\to\infty$ yields that $(z-y^*)^\top (w-y^*)\leq 0$. This concludes the proof.

\end{proof}

As a consequence of this lemma, we obtain the following result. A closed, convex set $G$ is called locally conic at $y\in G$ if and only if there exists $r>0$ such that $G\cap B(y,r)=(y+C)\cap B(y,r)$ where $C$ is the support cone to $G$ at $y$.

\begin{lemma} \label{lemma:diff-proj-almost-cone}
	Let $G\subseteq\R^d$ be a non-empty, closed, convex set and $x\in\R^d$. If $G$ is locally conic at $\pi_G(x)$, then $\pi_{G}$ has directional derivatives at $x$ given by 
	$$\diff^+\pi_G(x;\cdot)=\pi_{C_u}$$
	where $C$ is the support cone to $G$ at $\pi_G(x)$ and $u=x-\pi_G(x)$. 
\end{lemma}

\begin{proof}

	Set $y=\pi_G(x)$. Since we have, for all $z\in\R^d$, $\pi_G(z)=y+\pi_G(z-y)$ (this is easy to check using Lemma~\ref{lem:charact_proj} for instance), one may simply assume that $y=0$, without loss of generality. 
	
	Let $r>0$ be such that $G\cap B(0,r)=C\cap B(0,r)$. Let $z\in\R^d$. Since $G$ is locally conic at $y=0$ and $\pi_G$ is continuous (since it is non-expansive, see Lemma~\ref{lem:charact_proj}), $\pi_G(x+\varepsilon z)\in G\cap B(0,r)$ for all small enough $\varepsilon>0$. Hence, $\pi_G(x+\varepsilon z)=\pi_{G\cap B(0,r)}(x+\varepsilon z)=\pi_{C\cap B(0,r)}(x+\varepsilon z)$. 

Now, again using Lemma~\ref{lem:charact_proj}, we have that $\pi_C(x)=0$. Hence, again by continuity of $\pi_C$, it holds that for all small enough $\varepsilon>0$, $\pi_C(x+\varepsilon z)\in B(0,r)$, hence, $\pi_C(x+\varepsilon z)=\pi_{C\cap B(0,r)}(x+\varepsilon z)$ for such small enough $\varepsilon>0$. 

Finally, we have obtained that for all small enough $\varepsilon>0$, 
$$\pi_G(x+\varepsilon z)=\pi_C(x+\varepsilon z).$$
The conclusion follows using Lemma~\ref{lem:diff_proj}.

\end{proof}

An important case of locally conic convex sets is that of convex (possibly unbounded) polyhedra, that is, intersections of finitely many closed affine halfspaces. Indeed, we have the following lemma.

\begin{lemma} \label{lem:polyhedra}
	All convex polyhedras are locally conic at any point. 
\end{lemma}

As a consequence of this lemma, for all closed, convex polyhedra $G\subseteq\R^d$, $\pi_G$ has directional derivatives everywhere and $\diff^+\pi_G(x;\cdot)=\pi_{C_{x-\pi_G(x)}}$ for all $x\in\R^d$, where $C$ is the tangent cone (which coincides with the support cone) to $G$ at $x$. Note that this is also a particular case of Lemma~\ref{lem:shapiroSlater} above.

\begin{proof}
	Let $G$ be a convex polyhedra and $x\in\R^d$. If $x\notin G$, the result is vacuous, since the tangent cone to $G$ at $x$ is empty, as well as $G\cap B(0,r)$ for all small enough $r>0$. If $x\in\inter(G)$, the result is also trivial, since in that case, the tangent cone to $G$ at $x$ is $\R^d$. 
	
	Now, let $x\in\partial G$. Write $K=H_1\cap\ldots\cap H_p$ where $H_1,\ldots,H_p$ are closed affine halfspaces and $p\geq 1$ is an integer. Without loss of generality (or else, simply reorder $H_1,\ldots,H_p$), assume that $x\in \partial H_j$ for $j=1,\ldots,r$ and $x\notin \partial H_j$ for $j=r+1,\ldots,p$, for some $r\in\{1,\ldots,p\}$. That is, $H_1,\ldots,H_r$ are exactly those halfspaces whose bounding hyperplane contains $x$. Let $B=(H_1-x)\cap\ldots\cap (H_r-x)$. This is a closed, convex cone, as the intersection of closed, convex cones. Our goal is to show that $B$ coincides with $C$, the support cone to $G$ at $x$. Indeed, then, it is easy to see that $G\cap B(x,r)=(x+B)\cap B(0,r)$ for all small enough $r>0$: It suffices to take any $r\leq \min_{j\geq r+1}\dist(x,\partial H_j)$. 
	
For all $y\in K$, $y\in H_1\cap\ldots\cap H_p\subseteq H_1\cap\ldots\cap H_r$ so $y-x\in (H_1-x)\cap\ldots\cap (H_r-x)=B$. Hence, $B$ contains $G-x$ and, since $B$ is a closed, convex cone, it also contains $C$. Conversely, let $v\in B$ and let us show that $x+\varepsilon v\in G$ for some small enough $\varepsilon>0$. This will yield that $v\in C$. For all $j=1,\ldots,r$, $v\in H_j-x$ so $x+v\in H_j$. Now, by definition of $r$, $x\in\inter(H_j)$ for all $j=r+1,\ldots,p$, so there exists $\varepsilon>0$ such that $x+\varepsilon v\in H_j$ for all $j=r+1,\ldots,p$. Therefore, $x+\min(1,\varepsilon)v\in K$, yielding that $v\in C$ as desired. 
	
\end{proof}

Even though, in general, when $\pi_G$ has directional derivatives at some $x\in\R^d$, it does not necessarily hold that $\diff^+\pi_G(x;\cdot)=\pi_{C_{x-\pi_G(x)}}$, where $C$ is the support cone to $G$ at $\pi_G(x)$, the following result holds true.

\begin{lemma} \label{lem:diff_proj_sanity}

	Let $G\subseteq\R^d$ be a non-empty, closed, convex set and $x\in\R^d$. Let $C$ be the support cone to $G$ at $\pi_G(x)$. Assume that $\pi_G$ has directional derivatives at $x$. Then, for all $z\in\R^d$, $\diff^+\pi_G(x;z)\in C_{x-\pi_G(x)}$.
	
\end{lemma}

\begin{proof}

The fact that $\diff^+\pi_G(x;z)\in C$ is clear from the facts that, for all $\varepsilon>0$, $\varepsilon^{-1}(\pi_G(x+\varepsilon z)-\pi_G(x))\in C$, and $C$ is closed. Hence, we only need to show that $\diff^+\pi_G(x;z)$ is orthogonal to $x-\pi_G(x)$. For all $\varepsilon>0$, Lemma~\ref{lem:charact_proj} yields that 
$$(x+\varepsilon z-\pi_G(x+\varepsilon z))^\top (\pi_G(x)-\pi_G(x+\varepsilon z))\leq 0.$$
Using the fact that $\pi_G(x+\varepsilon z)=\pi_G(x)+\varepsilon \diff^+\pi_G(x;z)+o(\varepsilon)$ as $\varepsilon\to 0$, we obtain
$$-\varepsilon (x-\pi_G(x))^\top \diff^+\pi_G(x;z)+o(\varepsilon)\leq 0,$$
hence, by dividing by $\varepsilon$ and letting $\varepsilon\to 0$, $(x-\pi_G(x))^\top \diff^+\pi_G(x;z)\geq 0$. Moreover, again by Lemma~\ref{lem:charact_proj}, $(x-\pi_G(x))^\top \diff^+\pi_G(x;z)\leq 0$. Hence, we obtain orthogonality of $\diff^+\pi_G(x;z)$ with $x-\pi_G(x)$. 

\end{proof}

\begin{lemma} \label{lem:diff_proj_prop}

Let $G\subseteq\R^d$ be a non-empty, closed, convex set and $x\in\R^d$. Then, $\diff^+\pi_G(x;\cdot)$ is positively homogeneous and non-expansive with respect to $\|\cdot\|$. Moreover, for all $z,z'\in\R^d$,
$$\left(\diff^+\pi_G(x;z')-\diff^+\pi_G(x;z)\right)^\top (z'-z) \geq \|\diff^+\pi_G(x;z')-\diff^+\pi_G(x;z)\|^2\geq 0.$$ 

\end{lemma}

\begin{proof}

Positive homogeneity is clear from the definition. 

Using \eqref{eqn:charact_proj}, we have, for all $z,z'\in \R^d$, 
\begin{align*}
	\left(\diff^+\pi_G(x;z')-\diff^+\pi_G(x;z)\right)^\top (z'-z) & = \lim_{\varepsilon\downarrow 0} \varepsilon^{-1}\left( \pi_G(x+\varepsilon z')-\pi_G(x+\varepsilon z)\right)^\top (z'-z) \\
	& \geq \lim_{\varepsilon\downarrow 0} \varepsilon^{-1}\|\pi_G(x+\varepsilon z')-\pi_G(x+\varepsilon z)\|^2 \\
	& = \|\diff^+\pi_G(x;z')-\diff^+\pi_G(x;z)\|^2.
\end{align*}

Finally, non-expansiveness is a direct consequence of the last display, by using Cauchy--Schwarz inequality. 

\end{proof}

\begin{lemma} \label{lem:hardlemma}

Let $G\subseteq \R^d$ be a non-empty, closed, convex set. Fix $x\in\R^d$ and let $f(t)=\|\pi_G(tx)\|$, for all $t\geq 0$. Then, $f$ is non-decreasing and the map $t>0\mapsto f(t)/t$ is non-increasing. 

\end{lemma}

In other words, the norm of the projection is non-decreasing and has a non-increasing rate of change along any ray starting at $0$. By translating $G$ (noting that for all $x_0,x\in\R^d$, $x_0+\pi_G(x_0+x)=\pi_{G-x_0}(x)$), the lemma also applies to $f$ of the form $f(t)=\|x_0+\pi_G(x_0+tx)\|$ for any choice of $x_0,x\in\R^d$.

\begin{proof}

It is sufficient to show that for all $x\in\R^d$ and $t\geq 1$, 
$$\|\pi_G(x)\|\leq \|\pi_G(tx)\|\leq t\|\pi_G(x)\|.$$

First, Lemma~\ref{lem:charact_proj} yields the following two sets of inequalities:
\begin{equation} \label{eqn:proof111}
	(x-\pi_G(x))^\top (y-\pi_G(x))\leq 0, \quad \forall y\in G
\end{equation}
and
\begin{equation} \label{eqn:proof222}
	(tx-\pi_G(tx))^\top (z-\pi_G(tx))\leq 0, \quad \forall z\in G.
\end{equation}
Take $y=\pi_G(tx)$ and multiply \eqref{eqn:proof111} by $\lambda$, take $z=\pi_G(x)$ in \eqref{eqn:proof222} and sum the resulting inequalities:
$$(\pi_G(tx)-t\pi_G(x))^\top (\pi_G(tx)-\pi_G(x))\leq 0.$$
Expanding and using Cauchy--Schwarz inequality imply that
$$\|\pi_G(tx)\|^2+t\|\pi_G(x)\|^2-(t+1)\|\pi_G(x)\|\|\pi_G(tx)\|\leq 0.$$
Seeing this inequality as a second degree polynomial inequality in $\|\pi_G(tx)\|$ yields that 
$$\|\pi_G(x)\|\leq \|\pi_G(tx)\|\leq t\|\pi_G(x)\|$$
which is the desired result. 
\end{proof}

%
%
%
%
%

Finally, Lemma~\ref{lem:hardlemma} yields the following property of directional derivatives of projections.

\begin{lemma} \label{lem:monotonicity_dirder}

Let $G\subseteq\R^d$ be a non-empty, closed, convex set. Let $x\in\R^d$ and assume that $\pi_G$ has directional derivatives at $x$. Then, $\pi_G$ has directional derivatives at every point along the ray from $\pi_G(x)$ going through $x$, that is, at any point of the form $x_t:=\pi_G(x)+t(x-\pi_G(x))$, $t\geq 0$, and for all $s,t$ with $t>s>0$, and all $z\in\R^d$,
\begin{equation} \label{eqn:comparison_diff_proj}
	\|\diff^+\pi_G(x_t;z)\|\leq \|\diff^+\pi_G(x_s;z)\|.
\end{equation}
\end{lemma}

Note that $\pi_G$ automatically admits directional derivatives at $x_0=\pi_G(x)$, since $x_0\in G$.

\begin{proof}

The existence of directional derivatives at any $x_t,t>0$ follows from \cite[Proposition 2.2]{noll1995directional}. Following the proof of that proposition, we also obtain that for all $t>0$ and $z\in\R^d$, 
$$\diff^+\pi_G(x_t;z)=\diff^+\pi_G(x;A_t^{-1}(z))$$
where $A_t:\R^d\to\R^d$ is the bijective map defined as $A_t=tI_d+(1-t)\diff^+\pi_G(x;\cdot)$. 
In the rest of the proof, let us assume that $x_0=\pi_G(x)=0$, without loss of generality (we could simply translate $G$ without affecting the inequality that remains to be proven). Fix $s,t$ with $t>s>0$ and $z\in\R^d$. Then, for all $\varepsilon>0$, Lemma~\ref{lem:hardlemma} yields that
\begin{align*}
	\frac{\|\pi_G(x_t+\varepsilon z)\|}{\varepsilon} & = \frac{\|\pi_G(tx+\varepsilon z)\|}{\varepsilon} \\
	& = \frac{\|\pi_G(t(x+(\varepsilon/t) z))\|}{\varepsilon} \\
	& \leq \frac{t}{s}\frac{\|\pi_G(s(x+(\varepsilon/t) z))\|}{\varepsilon} \\
	& = \frac{\|\pi_G(x_s+(s\varepsilon/t) z))\|}{s\varepsilon/t}
\end{align*}
and taking the limit as $\varepsilon\to 0$ implies that $\|\diff^+\pi_G(x_t;z)\|\leq  \|\diff^+\pi_G(x_s;z)\|$.

\end{proof}

The following result allows to extend \eqref{eqn:comparison_diff_proj} to $s=0$. 

\begin{lemma} \label{lem:monotonicity_dirder2}

Let $G\subseteq\R^d$ be a non-empty, closed, convex set. Let $x\in\R^d$ and assume that $\pi_G$ has directional derivatives at $x$. Then, for all $z\in\R^d$, 
$$\|\diff^+\pi_G(x;z)\|\leq \|\diff^+\pi_G(\pi_G(x);z)\|.$$

\end{lemma}

Note that this lemma is not a consequence of \eqref{eqn:comparison_diff_proj} in Lemma~\ref{lem:monotonicity_dirder} because $s\geq 0\mapsto \|\diff^+\pi_G(x_s;z)\|$, for fixed $z\in\R^d$, is not always continuous at $s=0$. Take, for instance, $G=B(0,1)$, $x\in\R^d$ with $\|x\|>1$ and $z=-x$. 

\begin{proof}
	Without loss of generality, let us assume that $0\in G$ and $\pi_G(x)=0$. First, if $x\in G$, then $\pi_G(x)=x$ and the result is trivial. Assume that $x\notin G$. Fix $z\in\R^d$ and let $\varepsilon>0$. Then, we have
\begin{align}
	\|\pi_G(x+z)\|^2 & = (x+z)^\top \pi_G(x+z)-(x+z-\pi_G(x+z))^\top \pi_G(x+z) \nonumber \\
	& \leq (x+z)^\top \pi_G(x+z) \quad \mbox{ by Lemma~\ref{lem:charact_proj}} \nonumber \\
	& \leq z^\top \pi_G(x+z) \quad \mbox{ by Lemma~\ref{lem:charact_proj}, noting that } \pi_G(x)=0 \nonumber \\
	& = z^\top \pi_G(z)+z^\top(\pi_G(x+z)-\pi_G(z)) \nonumber \\
	& \leq z^\top \pi_G(z)+\pi_G(z)^\top(\pi_G(x+z)-\pi_G(z)) \quad \mbox{ by Lemma~\ref{lem:charact_proj}}\nonumber \\
	& = (z-\pi_G(z))^\top \pi_G(z)+\pi_G(z)^\top\pi_G(x+z). \label{eqn:hard45}
\end{align}
Now, replacing $z$ with $\varepsilon z$ in \eqref{eqn:hard45}, dividing by $\varepsilon^2$ and letting $\varepsilon\downarrow 0$, we obtain that
\begin{equation} \label{eqn:hard55}
	\|\diff^+\pi_G(x;z)\|^2\leq (z-\diff^+\pi_G(0;z))^\top \diff^+\pi_G(0;z)+\diff^+\pi_G(0;z)^\top \diff^+\pi_G(x;z).
\end{equation}
Note that $\pi_G$ has directional derivatives at $0$ since we have assumed that $0\in G$. Moreover, since $x\notin G$, $0=\pi_G(x)$ must be on the boundary of $G$. Hence, by Lemma~\ref{lem:Zarantonello}, $\diff^+\pi_G(0;\cdot)=\pi_C$ where $C$ is the support cone to $G$ at $0$. Therefore, $(z-\diff^+\pi_G(0;z))^\top \diff^+\pi_G(0;z)=(z-\pi_C(z))^\top \pi_C(z)$. Now, note that by Lemma~\ref{lem:charact_proj}, for all $y\in C$, we have that 
$$(z-\pi_C(z))^\top (y-\pi_C(z))\leq 0.$$
Taking $y=0$ on the one hand, and $y=2\pi_C(z)$ on the second hand, yields that $(z-\pi_C(z))^\top \pi_C(z)= 0$. Therefore, continuing \eqref{eqn:hard55}, we obtain that 
$$\|\diff^+\pi_G(x;z)\|^2\leq +\diff^+\pi_G(0;z)^\top \diff^+\pi_G(x;z)$$
which is bounded by $\|\diff^+\pi_G(0;z)\|\|\diff^+\pi_G(x;z)\|$ by Cauchy--Schwarz inequality. The desired result follows readily.
\end{proof}

%
%
%


%
%
%

\section{Adaptation of the convergence results for $U$-estimators} \label{sec:appendix_extra_proofs}

In this section, we assume that the loss function $\phi:E^k\times\Theta_0\to\R$ is symmetric and measurable in its first $k$ arguments and convex in its last, and that for all $\theta\in\Theta_0$, $\phi(\cdot,\theta)\in L^1(P^{\otimes k})$. This allows to define the population risk $\Phi(\theta)=\E[\phi(X_1,\ldots,X_k,\theta)]$ for all $\theta\in\Theta_0$. For all $n\geq k$, we define the empirical risk $\Phi_n$ as $\Phi_n(\theta)=\frac{1}{{n\choose k}}\sum_{1\leq i_1<\ldots<i_k\leq n}\phi(X_{i_1},\ldots,X_{i_k},\theta)$, for all $\theta\in\Theta_0$. 

For simplicity, for every subset $I\subseteq\{1,\ldots,n\}$ of size $k$, we denote by $X_I$ the vector $(X_{i_1},\ldots,X_{i_k})$ where $i_1<\ldots<i_k$ are the elements of $I$ ordered in increasing order. We also denote by $\mathcal P_{k,n}$ the collection of all subsets of size $k$ of $\{1,\ldots,n\}$.

\subsection{Non-differentiable case}

Let us prove the analog of Proposition~\ref{prop:non-smooth} for $U$-estimators. Analogs of Theorems~\ref{thm:non-diff} and Theorems~\ref{thm:First_order_asymp} will follow directly. 

\begin{proposition} \label{prop:LLN_U}
	Assume that $\phi(\cdot,\theta)\in L^2(P^{\otimes k})$ for all $\theta\in\Theta_0$. Let $(\rho_n)_{n\geq 1}$ be any non-decreasing sequence of positive numbers diverging to $\infty$ as $n\to\infty$. Then, for all $\theta\in\Theta_0$ and $t\in\R^d$, 
	$$\rho_n(\Phi_n(\theta+t/\rho_n)-\Phi_n(\theta))\xrightarrow[n\to\infty]{} h_{\partial\Phi(\theta)}(t)$$
in probability.

\end{proposition}

\begin{proof}
	Similarly to the proof of Proposition~\ref{prop:non-smooth}, fix $t\in\R^d$ and define 
\begin{align*}
F_n(t) & = \rho_n\left(\Phi_n(\theta+t/\rho_n)-\Phi_n(\theta)-\frac{1}{{n\choose k}\rho_n}t^\top \sum_{I\in\mathcal P_{n,k}} g(X_I,\theta)\right) \\
& \quad\quad\quad -\rho_n\left(\Phi(\theta+t/\rho_n)-\Phi(\theta)-\frac{1}{\rho_n}t^\top \E[g(X_1,\ldots,X_k,\theta)]\right)
\end{align*}
and write $\DS F_n(t)=\sum_{I\in\mathcal P_{n,k}}(Z_{I,n}-\E[Z_{I,n}])$
where we set 
$$Z_{I,n}=\frac{\rho_n}{{n\choose k}}\left(\phi(X_I,\theta+t/\rho_n)-\phi(X_I,\theta)-(1/\rho_n)t^\top g(X_I,\theta)\right)$$ 
for all $I\in\mathcal P_{n,k}$.
Lemma~\ref{lem:monoton_subgrad} yields that 
\begin{align*}
	0 \leq Z_{I,n} & \leq \frac{1}{{n\choose k}}t^\top(g(X_I,\theta+t/\rho_n)-g(X_I,\theta)) \\ 
	& \leq \frac{1}{{n\choose k}}t^\top(g(X_I,\theta+t/\rho_1)-g(X_I,\theta))
\end{align*}
for all $I\in\mathcal P_{n,k}$. Denoting by $Y_I=t^\top(g(X_I,\theta+t/\rho_1)-g(X_I,\theta))$ for all $I\in\mathcal P_{n,k}$, we obtain that for all large enough $n$ ($n\geq 2k$ suffices)
\begin{align*}
	\var(F_n(t)) & = \var\left(\sum_{I\in \mathcal P_{n,k}} Z_{I,n}\right) = \sum_{I,J\in \mathcal P_{n,k}, I\cap J\neq\emptyset} \cov(Z_{I,n},Z_{J,n}) \\
	& \leq \frac{1}{{n\choose k}^2}\sum_{I,J\in \mathcal P_{n,k}, I\cap J\neq\emptyset} \E[Y_{I}Y_{J}] \\ 
	& = \frac{1}{{n\choose k}^2}\sum_{j=1}^k {n\choose k}{k\choose j}{{n-k}\choose {k-j}}\alpha_j \\
	& = \frac{1}{{n\choose k}}\sum_{j=1}^k {k\choose j}{{n-k}\choose {k-j}}\alpha_j
\end{align*}
where $\alpha_j=\E[Y_IY_J]$ for any two sets $I,J\in\mathcal P_{n,k}$ with $\#(I\cap J)=j$, $j=1,\ldots,k$ ($\#A$ stands for the cardinality of a set $A$). In the second equality, we used the fact that $Z_{I,n}$ and $Z_{J,n}$ are independent if $I\cap J=\emptyset$. In the second to last equality, we used the fact that for $j=1,\ldots,k$, the number of pairs of sets $I,J\in\mathcal P_{n,k}$ with $\#(I\cap J)=j$ is ${n\choose k}{k\choose j}{{n-k}\choose {k-j}}$ (choose $I$ first, then $j$ elements in $I$ and $k-j$ outside of $I$ to obtain $J$). 

Note that $\alpha_1,\ldots,\alpha_k$ do not depend on $n$, and each term in the product is of order at most $1/n$. Hence, $\var(F_n(t))\xrightarrow[n\to\infty]{}0$ so $F_n(t)\xrightarrow[n\to\infty]{}0$ in probability. 

By Theorem~\ref{thm:LLN-U}, $\DS \frac{1}{{n\choose k}} \sum_{I\in\mathcal P_{n,k}} g(X_I,\theta) \xrightarrow[n\to\infty]{}\E[g(X_1,\ldots,X_k,\theta)]$ almost surely, hence, in probability, and Lemma~\ref{lem:dirder_support} yields that $\rho_n\left(\Phi(\theta+t/\rho_n)-\Phi(\theta)\right)\xrightarrow[n\to\infty]{}h_{\partial\Phi(\theta)}(t)$. Hence, we obtain the desired result. 

\end{proof}

\subsection{Proof of Theorem~\ref{sec:appendix_extra_proofs}}

As in the proof of Theorem~\ref{thm:asympt_norm}, fix $R>0$ and let 
$$F_n(t)=n\left(\Phi_n(\theta^*+t/\sqrt n)-\Phi_n(\theta^*)\right)-\left(\frac{\sqrt n}{{n\choose k}}t^\top\sum_{I\in\mathcal P_{n,k}}g(X_I,\theta^*)+\frac{1}{2}t^\top St\right)$$
for all $t\in B_S(0,R)$, for all large enough $n$ so $B_S(\theta^*,R/\sqrt n)\subseteq \Theta_0$, and where $S=\nabla^2\Phi(\theta^*)$. Let us show that for all $t\in B_S(0,R)$, $F_n(t)\xrightarrow[n\to\infty]{}0$ in probability. For this, we let 
$$Z_{I,n}=\frac{n}{{n\choose k}}\left(\phi(X_I,\theta^*+t/\sqrt n)-\phi(X_I,\theta^*)-\frac{t^\top}{\sqrt n}g(X_I,\theta^*)\right)$$
for each $I\in\mathcal P_{n,k}$.
Now, we note that for each $I\in\mathcal P_{n,k}$,
$$0\leq Z_{I,n}\leq \frac{\sqrt n}{{n\choose k}}\left(g(X_I,\theta^*+t/\sqrt n)-g(X_I,\theta^*)\right)$$
thanks to Lemma~\ref{lem:monoton_subgrad}. Setting $Y_{I,n}=g(X_I,\theta^*+t/\sqrt n)-g(X_I,\theta^*)$ for all $I\in\mathcal P_{n,k}$, we obtain:
\begin{align*}	\var\left(\sum_{I\in\mathcal P_{n,k}}Z_{I,n}\right) & = \sum_{I,J\in\mathcal P_{n,k}, I\cap J\neq\emptyset} \cov(Z_{I,n},Z_{J,n}) \\
	& \leq \frac{n}{{n\choose k}^2} \sum_{I,J\in\mathcal P_{n,k}, I\cap J\neq\emptyset} \E[Y_{I,n}Y_{J,n}] \\
	& = \frac{n}{{n\choose k}} \sum_{j=1}^k {k\choose j}{{n-k}\choose {k-j}} a_{j,n}
\end{align*}
where, for all $j=1,\ldots,k$, $a_{j,n}=\E[Y_{I,n}Y_{J,n}]$ for any fixed $I,J\in\mathcal P_{n,k}$ with $\#(I\cap J)=j$. Fix $I_0=\{1,\ldots,k\}$ and $J_0=\{1,\ldots,j,k+1,\ldots,2k-j\}$. Now, just as in the proof of Theorem~\ref{thm:asympt_norm}, note that $(Y_{I_0,n})_{n\geq 1}$ is a non-increasing sequence (by Lemma~\ref{lem:monoton_subgrad}) of non-negative random variables, hence, it converges almost surely to some non-negative random variable $Y_{I_0}$. By monotone convergence, we must then have that $\E[Y_{I_0,n}]\xrightarrow[n\to\infty]{} \E[Y_{I_0}]$. Yet, $\E[Y_{I_0,n}]=\nabla\Phi(\theta^++t/\sqrt n)-\nabla\Phi(\theta^*)$, which goes to $0$ as $n\to\infty$, since $\Phi$ is twice differentiable at $\theta^*$, hence $\nabla\Phi$ is continuous at $\theta^*$. Therefore, $\E[Y_{I_0}]=0$, which yields that $Y_{I_0}=0$ almost surely. Similarly, $Y_{J_0,n}\xrightarrow[n\to\infty]{} 0$ almost surely and, now, monotone convergence implies that $\E[Y_{I_0,n}Y_{J_0,n}]\xrightarrow[n\to\infty]{}0$. 
Finally, we obtain that each $a_{j,n}\xrightarrow[n\to\infty]{}0$, $j=1,\ldots,k$. Moreover, for each $j=1,\ldots,k$, ${k\choose j}{{n-k}\choose{k-j}}$ is of the same order as $n^j$ as $n\to\infty$ so we readily obtain that 
$$\var\left(\sum_{I\in\mathcal P_{n,k}}Z_{I,n}\right) \xrightarrow[n\to\infty]{}0. $$
The fact that $F_n(t)\xrightarrow[n\to\infty]{}0$ in probability then follows from $\phi$ being twice differentiable at $\theta^*$. 

Now, the rest of the proof is almost identical to that of Theorem~\ref{thm:asympt_norm}, with the only difference that, using Theorem~\ref{thm:CLT_U}, an extra $k$ factor will appear in the asymptotic behavior of the minimizer $t_n^*$ of $\DS \frac{\sqrt n}{{n\choose k}}t^\top\sum_{I\in\mathcal P_{n,k}}g(X_I,\theta^*)+\frac{1}{2}t^\top St$.

\section{Miscellaneous results} \label{sec:misc}

Let us give yet a second corollary to Lemma~\ref{lemma:Rockafellar}, that allows to go from pointwise to uniform convergence, in $L^p$ sense ($p\geq 1$).

\begin{corollary} \label{cor:Rockafellar_Lp}
	Let $p\geq 1$. Let $f,f_1,f_2,\ldots$ be random convex functions defined on an open convex set $G_0\subseteq \R^d$. Assume that for all $n\geq 1$ and all $t\in G_0$, $f_n(t)\in L^p(\PP)$. Assume also that
	$\E[|f_n(t)-f(t)|^p]\xrightarrow[n\to\infty]{} 0$ for all $t\in G_0$. Then, for all compact sets $K\subseteq G_0$, $\E[\sup_K|f_n-f|^p]\xrightarrow[n\to\infty]{} 0.$ 
\end{corollary}

\begin{proof}

	Let $K\subseteq G_0$ be a compact set. Since $K$ is compact, so is its convex hull, by \cite[Theorem 1.1.11]{schneider2014convex}. Hence, without loss of generality, in the sequel, let us assume that $K$ is convex.

Since $G_0$ is open, there exists $\eta>0$ satisfying that $K^{2\eta}:=\{x\in\R^d:\dist(x,K)\leq 2\eta\}\subseteq G_0$. Moreover, there exists a convex polytope $P$ with $K^{\eta}\subseteq P\subseteq K^{2\eta}$, see \cite{bronshteyn1975approximation}. 
Let $v_1,\ldots,v_r$ ($r\geq 1$) the vertices of $P$. 

Fix $\varepsilon>0$ with $\varepsilon\leq \eta/2$ and let $t_1,\ldots,t_N\in K$ (with $N\geq 1$) be an $\varepsilon$-approximation of $K$, that is, such that for all $t\in K$, $\|t-t_j\|\leq \varepsilon$ for some $j\in\{1,\ldots,N\}$. 

Let $t\in K$ and $j\in\{1,\ldots,N\}$ satisfying $\|t-t_j\|\leq \varepsilon$. Assume for now that $t\neq t_j$ and let $z_-$ and $z_+$ be the two points at the intersection of $\partial P$ and the line passing through $t$ and $t_j$, that is, 
$$z_-=t_j+\lambda_-(t-t_j)$$
and
$$z_+=t+\lambda_+(t_j-t)$$
for some $\lambda_-,\lambda_+\geq 1$. Note that
$\lambda_-=\frac{\|z_--t_j\|}{\|t-t_j\|}=\frac{\|z_--t\|}{\|t-t_j\|}+1\geq 1+\frac{\eta}{\varepsilon}>1$ and that $\lambda_+= \frac{\|z_+-t\|}{\|t-t_j\|}=\frac{\|z_+-t_j\|}{\|t-t_j\|}+1\geq 1+\frac{\eta}{\varepsilon}>1$. For each $n\geq 1$, convexity of $f_n$ yields, on the one hand, that 
\begin{equation} \label{eqn:proc86}
f_n(t)-f_n(t_j) \leq (1/\lambda_-)\left(f_n(z_-)-f_n(t_j)\right) \leq (1/\lambda_-)\left(\max_P f_n-f_n(t_j)\right) \leq \frac{\varepsilon}{\eta}\left(\max_P f_n-f_n(t_j)\right)
\end{equation}
and that 
$$f_n(t_j) \leq (1-1/\lambda_+)f_n(t)+(1/\lambda_+)f_n(z_+),$$
so, dividing both sides by $(1-1/\lambda_+)$ and subtracting $f_n(t_j)$, 
\begin{align} 
f_n(t)-f_n(t_j) & \geq (1/\lambda_+)(1-1/\lambda_+)^{-1}\left(f_n(t_j)-f_n(z_+)\right) \nonumber \\
	& = \frac{1}{\lambda_+-1}\left(f_n(t_j)-f_n(z_+)\right) \nonumber \\
	& \geq \frac{1}{\lambda_+-1}\left(f_n(t_j)-\max_P f_n\right) \nonumber \\
	& \geq \frac{\varepsilon}{\eta}\left(f_n(t_j)-\max_P f_n\right). \label{eqn:proc87}
\end{align}

Let $M_n=\max(f_n(v_1),\ldots,f_n(v_r))$ and $m_n=\min(f_n(t_1),\ldots,f_n(t_N))$. Convexity of $f_n$ yields that $\max_P f_n\leq M_n$ and we obtain, from \eqref{eqn:proc86} and \eqref{eqn:proc87}, that
\begin{equation*}
	|f_n(t)-f_n(t_j)|\leq \frac{\varepsilon}{\eta}(M_n-m_n).
\end{equation*}
Similarly for $f$, we have that 
\begin{equation*}
	|f(t)-f(t_j)| \leq \frac{\varepsilon}{\eta}(M-m)
\end{equation*}
where we set $M=\max(f(v_1),\ldots,f(v_p))$ and $m=\min(f(t_1),\ldots,f(t_N))$.

Finally, writing $|f_n(t)-f(t)|\leq |f_n(t)-f_n(t_j)|+|f_n(t_j)-f(t_j)|+|f(t)-f(t_j)|$, we obtain:
\begin{equation*}
	\sup_{t\in K}|f_n(t)-f(t)| \leq \frac{\varepsilon}{\eta}(M_n-m_n+M-m)+\max_{1\leq j\leq N}|f_n(t_j)-f(t_j)|.
\end{equation*}
Now, raising to the power $p$ and taking the expectation on both sides, we obtain that 
\begin{align*}
	\E[\sup_{t\in K}|f_n(t)-f(t)|^p] & \leq \frac{5^{p-1}\varepsilon}{\eta}\left(\E[|M_n|^p]+\E[|m_n|^p]+\E[|M|^p]+\E[|m|^p]\right) \\
	& \quad \quad \quad \quad +5^{p-1}\sum_{j=1}^N\E\left[|f_n(t_j)-f(t_j)|^p\right]
\end{align*}
where we used the fact that $(a_1+a_2+a_3+a_4+a_5)^p\leq 5^{p-1}(a_1^p+a_2^p+a_3^p+a_4^p+a_5^p)$ for all positive numbers, $a_1,a_2,a_3,a_4,a_5$ and we bounded the maximum of non-negative numbers by their sum in the last term. Now, the assumption implies that, for large enough, each term in the last sum can be bounded by $\varepsilon/N$ and, hence, the whole right hand side can be bounded by $C\varepsilon$ for some positive constant $C$. Since $\varepsilon>0$ was any (small enough) positive number, this implies that $\E\left[\sup_{t\in K}|f_n(t)-f(t)|^p\right]\xrightarrow[n\to\infty]{}0$. 
\end{proof}

\begin{lemma} \label{lem:convergence_proba}
Let $(Z_n)_{n\geq 1}$ be a sequence of real random variables satisfying that $\rho_nZ_n\xrightarrow[n\to\infty]{} 0$ in probability, for any choice of non-decreasing sequence $(\rho_n)_{n\geq 1}$ of positive numbers, diverging to $\infty$ as $n\to\infty$. Then, $P(Z_n\neq 0)\xrightarrow[n\to\infty]{}0$. 
\end{lemma}

\begin{proof}
Assume, for the sake of contradiction, that $P(Z_n\neq 0)$ does not go to $0$ as $n\to\infty$. That is, there is some $\varepsilon$ and an increasing sequence $(k_n)_{n\geq 1}$ of positive integers such that $P(Z_{k_n}\neq 0)\geq \varepsilon$ for all $n\geq 1$. For each $n\geq 1$, since the map $t\in\R\mapsto P(|Z_{k_n}|> t)$ is right-continuous, so there must exist some $\alpha_n>0$ with the property that $P(|Z_{k_n}|> \alpha_n)\geq\varepsilon/2$. Since $Z_n\xrightarrow[n\to\infty]{} 0$ in probability by assumption, the sequence $(\alpha_n)_{n\geq 1}$ must converge to $0$. Since we could extract a non-increasing subsequence of it, let us assume, for simplicity, that $(\alpha_n)_{n\geq 1}$ is non-increasing. Then, one can choose a non-decreasing sequence $(\rho_n)_{n\geq 1}$ of positive numbers, such that $\rho_{k_n}=1/\alpha_n$ for all $n\geq 1$. The assumption implies that $\rho_{k_n}Z_{k_n}$ must converge to $0$ in probability, as a subsequence of $(\rho_nZ_n)_{n\geq 1}$. Since this is not the case by construction, we obtain a contradiction. 
\end{proof}

%
%
%

\end{document}